\documentclass{amsart}
\usepackage{amssymb}
\usepackage{graphicx}

\usepackage{amsmath,amscd}

\usepackage{amsthm}

\title[Abstract homomorphisms from some topological groups]{Abstract homomorphisms from some topological groups to acylindrically hyperbolic groups}

\newtheorem{theorem}{Theorem}

\newtheorem{bigtheorem}{Theorem}

\numberwithin{theorem}{section}
\newtheorem{corollary}[theorem]{Corollary}
\newtheorem{proposition}[theorem]{Proposition}
\newtheorem{lemma}[theorem]{Lemma}
\theoremstyle{definition}\newtheorem{definition}[theorem]{Definition}
\theoremstyle{definition}
\theoremstyle{definition}\newtheorem{example}[theorem]{Example}
\theoremstyle{definition}\newtheorem{remark}[theorem]{Remark}
\theoremstyle{definition}
\theoremstyle{definition}
\theoremstyle{definition}
\theoremstyle{definition}
\theoremstyle{definition}\newtheorem{notation}[theorem]{Notation}
\theoremstyle{definition}\newtheorem{constants}[theorem]{Constants}

\newcommand{\HEG}{\operatorname{HEG}}
\newcommand{\Ell}{\operatorname{Ell}}
\newcommand{\Lox}{\operatorname{Lox}}
\newcommand{\Rad}{\operatorname{Rad}}
\newcommand{\CRad}{\operatorname{CRad}}
\newcommand{ \Rel}{\operatorname{Rel}}
\newcommand{\MCG}{\operatorname{Mod}}

\newcommand{\PMCG}{\operatorname{PMod}}
\newcommand{\C}{\operatorname{C}}

\newcommand{\SL}{\operatorname{SL}}
\newcommand{\Out}{\operatorname{Out}}

\newcommand{\im}{\operatorname{im}}

\begin{document}

\author{Oleg Bogopolski}
\address{{Sobolev Institute of Mathematics of Siberian Branch of Russian Academy
of Sciences, Novosibirsk, Russia}\newline
{and D\"{u}sseldorf University, Germany}}
\email{Oleg$\_$Bogopolski@yahoo.com}

\author[Samuel M. Corson]{Samuel M. Corson}
\address{Instituto de Ciencias Matem\'aticas CSIC-UAM-UC3M-UCM, 28049 Madrid, Spain.}
\email{sammyc973@gmail.com}

\thanks{The work of the second author is supported by European Research Council grant PCG-336983 and by the Severo Ochoa Programme for Centres of Excellence in R\&D SEV-20150554.}

\begin{abstract}
We describe homomorphisms $\varphi:H\rightarrow G$ for which the codomain is acylindrically hyperbolic and the domain is a topological group which is either completely metrizable or locally countably compact Hausdorff.  It is shown that, in a certain sense,
either the image of $\varphi$ is small or $\varphi$ is almost continuous.
We also describe homomorphisms from the Hawaiian earring group to $G$ as above.

We prove a more precise result for homomorphisms $\varphi:H\rightarrow \MCG(\Sigma)$, where $H$ as above and
$\MCG(\Sigma)$ is the mapping class group of a connected compact surface $\Sigma$. In this case there exists an open normal subgroup $V\leqslant H$ such that $\varphi(V)$ is finite.
We also prove the analogous statement for homomorphisms $\varphi:H\rightarrow \Out(G)$, where $G$ is a one-ended hyperbolic group.

Some automatic continuity results for relatively hyperbolic groups and fundamental groups of graphs of groups  are also deduced.
As a by-product, we prove that the Hawaiian earring group is acylindrically hyperbolic,
but does not admit any universal acylindrical action on a hyperbolic space.

%We prove combination theorems for these groups with respect to the property of automatic continuity.
\end{abstract}

\maketitle

\setcounter{tocdepth}{1}
\tableofcontents

\begin{section}{Introduction}
It is natural to ask, for which commonly interesting classes $\mathcal{A}$ and $\mathcal{B}$ of topological groups any abstract homomorphism $\varphi: A\rightarrow B$, where $A\in \mathcal{A}$ and $B\in \mathcal{B}$,
is continuous, has finite image, or satisfies a weaker variation of these properties.
%Farb and Masur~\cite{Farb_Masur} proved a variation of Margulis' superrigidity theorem~\cite{Margulis},
%where the domain of $\varphi$ as above and the target group is the mapping class group of a compact surface;
%they proved that in this case the image of $\varphi$ is always finite. Fujuwara~\cite{Fujiwara} generalized
%their result for the outer automorphism group of a one-ended hyperbolic group.
%Bader and Furman~\cite{Bader_Furman}
%generalized Margulis' theorem, by replacing the domain of $\varphi$ by a closed subgroup in a
%certain locally compact second countable group.
In many interesting cases one can see that, informally speaking,

\medskip

\centerline {\it either the image of $\varphi$ is small or $\varphi$ is almost continuous.}

\medskip

%Moreover, one can make these notions precise depending on the classes of groups under consideration.
% with appropriately %where these notions can be defined precisely.

So the famous superrigidity theorem of Margulis~\cite{Margulis} says
that if $G$ is a real semisimple Lie group $G$ of rank at least 2 with finite center and no nontrivial compact factors,
$\Gamma$ is an irreducible lattice in $G$, and $H$ is a simple Lie group,
then any homomorphism $\varphi:\Gamma\rightarrow H$ has finite image or uniquely extends
to a continuous homomorphism $\widetilde{\varphi}:G\rightarrow H$.
%that any abstract homomorphism from an irreducible lattices of certain semisimple Lie groups
%to simple Lie groups either has finite image or uniquely extends
%to a continuous homomorphism from the ambient Lie group.
Some variations of this theorem for non-Lie type groups are given by Farb and Masur~\cite{Farb_Masur},  Fujiwara~\cite{Fujiwara}, Bader and Furman~\cite{Bader_Furman}.

Another important result in this direction is a theorem of Nikolov and Segal~\cite[Theorem 1.1]{NS},
which says that every abstract homomorphism from a finitely generated in topological sense profinite group to any profinite group is continuous.
Papers~\cite{Du, MN, BHK, CK,PS, KV, CC,CK} deal with automatic continuity of abstract homomorphisms from locally compact Hausdorff groups to some discrete groups; papers~\cite{CC,CK} also deal with completely metrizable groups as domains.
%These papers do not cover the case, where the domain of $\varphi$ is an arbitrary {\it countably} locally compact group and the target group of $\varphi$ is an arbitrary acylindrically hyperbolic group.

One of the novelties of this paper is that we describe homomorphisms $\varphi$ whose domain is an arbitrary completely metrizable or {\it countably} locally compact Hausdorff group and whose target group is an arbitrary acylindrically hyperbolic group (in fact we consider larger classes of domains and targets).

Recall that a space is {\it countably compact} if every countable open cover admits a finite subcover, and a space $X$ is {\it locally countably compact}
if for each point $x\in X$ there exists a countably compact neighborhood of $x$ in $X$.
%\marginpar{\tiny Each discrete group $G$ is locally compact. Therefore it is better to say that the structure
%of $G^{\circ}$ is unclear, or that the algebraic-topological structure of $G$ is unclear.}
Clearly the class $\bf{LCC}$ of locally countably compact Hausdorff groups includes the class $\bf{LC}$ of locally compact Hausdorff groups, but the inclusion is strict~\cite{HvRS}.  The class of locally compact groups has been studied for several decades and is well understood (see, for example, the book~\cite{Stroppel}).  By contrast, almost nothing is known regarding the algebraic-topological structure of a locally countably compact Hausdorff group.
It seems that the structure of such groups can be very complicated: there exist two countably compact Hausdorff topological groups
whose direct product is not countably compact ~\cite{HvRS}.
%He asks, whether groups with this property exist if we assume only ZFC.
%Therefore none of the classical structure theorems for locally compact groups
%can be used in the proof of Theorem \ref{metricandlcH}.

%In case a topological group is either completely metrizable or locally countably compact Hausdorff
%we obtain an analogous theorem:

The acylindrically hyperbolic groups- those groups which admit a non-elementary acylindrical action on a hyperbolic space- have become a focus of attention in geometric group theory.  The class of such groups is large enough to admit a great number of groups of classical interest.  Examples include
non-(virtually cyclic) groups that are hyperbolic relative to proper subgroups,
many 3-manifold groups, many groups acting on trees,
non-(virtually cyclic) groups acting properly on proper CAT(0)-spaces and containing rank-one elements, non-cyclic directly indecomposable right-angled Artin groups,
all but finitely many mapping class groups,
groups of deficiency at least~2,
$\Out(F_n)$ for $n\geqslant 2$ (see \cite{Osin_1, Osin_3} for references and historical remarks).

Suppose that a group $G$ acts isometrically on a metric space $S$.
An element $g\in G$ is called {\it elliptic} for this action if some (equivalently every)
orbit of $\langle g\rangle$ is bounded. We introduce the elliptic radical of an arbitrary group $G$.
%Let $\Ell(G\curvearrowright S)$ be the set of elliptic elements of $G$ for the given action of $G$ on $S$.

\begin{definition}
%Let $G$ be an arbitrary group.
The {\it elliptic radical} of $G$, denoted by $\Rad_{\Ell}(G)$, is the set of all elements $g\in G$ which
are elliptic for any acylindrical action of $G$ on any hyperbolic space.
\end{definition}

Note that the elliptic radical of an acylindrically hyperbolic group is, in a certain sense, relatively small
(see Remark~\ref{EllipticRadicalSmall}).
For example, the elliptic radical of a relatively hyperbolic group is contained in the union of the set of elements of finite order and the set of parabolic elements (see Example~\ref{Rad_Rel_Hyp}).
%The elliptic radical of the mapping class group of a closed compact orientable surface
%of genus at least 2 consists of elements of finite order and reducible elements.
%\marginpar{\tiny cite}

Informally, the main Theorem~\ref{metricandlcH} says that
abstract homomorphisms from certain topological groups to acylindrically hyperbolic groups
either have relatively small images or are almost continuous.

\begin{bigtheorem}\label{metricandlcH}
%\marginpar{\tiny Using arguments of Kramer (proofs of his C,D) in the case of locally compact groups,
%one can choose $V$ to be an open subgroup containing $H^{\circ}$.}
Let $H$ be a topological group which is either completely metrizable or locally countably compact Hausdorff.
Let $\varphi:H\rightarrow G$ be an abstract homomorphism, where $G$ is an arbitrary group.
Then either $\varphi(H)$ is contained in the elliptic radical $\Rad_{\Ell}(G)$,
or there exists a normal open subgroup $V\leqslant H$ with finite $\varphi(V)$.
%with finite $\varphi(V)$.
%Let $G$ be a group which acts coboundedly and acylindrically on a hyperbolic space $S$.
%Then for any abstract group homomorphism $\varphi:H \rightarrow G$
%either $\varphi(H)$ acts elliptically on $S$, or there exists a normal open subgroup $V\leqslant H$
%with finite $\varphi(V)$.
%%such that $\varphi(V)$ is a subset of $G$ consisting of only elliptic elements with respect to this action.
\end{bigtheorem}

Theorem~\ref{generalisation_of_B}, generalizing Theorem~\ref{metricandlcH}, says that an analogous statement is  valid for all finite index subgroups in completely metrizable or locally countably compact Hausdorff topological groups.

The second main theorem is Theorem~\ref{thebigone}. It describes homomorphisms from the Hawaiian earring group
to acylindrically hyperbolic groups, and its proof is used in the proof of Theorem~\ref{metricandlcH}.
Definitions of the Hawaiian earring group $\HEG$ and its subgroups $\HEG^n$ are given in Section~2.

\begin{bigtheorem}\label{thebigone}
Let $G$ be a group which acts coboundedly and acylindrically on a hyperbolic space $S$. Then for any abstract group homomorphism $\varphi:\HEG\rightarrow G$, there exists a natural number $n$ such that $\varphi (\HEG^n)$ is a subgroup of $G$ consisting of only elliptic elements with respect to this action.
\end{bigtheorem}

%Let $\mathcal{A}(G)$ be the class of all acylindrical actions of $G$ on hyperbolic spaces.
%We define the {\it elliptic radical} of $G$ by
%$$
%\Rad_{\Ell}(G)=\underset{(G\curvearrowright S)\in \mathcal{A}(G)}{\bigcap} \Ell(G\curvearrowright S).
%$$

%Using a result of Hajnal and Juh$\acute{\rm a}$sz~\cite{Hajnal_Juhasz},
%One can have, for example, the $H$ in the hypotheses be a locally compact Hausdorff group.
%However (under ZFC + Countinuum hypothesis)

%\medskip

Results like Theorems~\ref{metricandlcH} and~\ref{thebigone} are called \textit{atomic properties} in the literature (see \cite{Ed2} and \cite{N}).
Note that in Theorems~~\ref{metricandlcH} and~\ref{thebigone} we are not requiring the group $G$
to be acyllindrically hyperbolic.  For example any group acts coboundedly and acylindrically on a point,
but the set of elliptic elements with respect to this action is all of $G$.

\begin{remark}
The conclusions in Theorems~\ref{thebigone} and~\ref{metricandlcH} cannot be strengthened to say that $\ker(\varphi)$ is open, as seen by easy examples below.

The free product $G = F_2 \ast (\mathbb{Z}/2\mathbb{Z})$ of the free group of rank $2$ and the group of order $2$ is acylindrically hyperbolic.  There exists a homomorphism from $\HEG$ to $\mathbb{Z}/2\mathbb{Z}$ whose kernel does not contain any $\HEG^n$ (see, for example, \cite{CS}), hence the kernel is not open, and so such a homomorphism exists to $G$.

Also, by endowing the direct product $\prod_{\mathbb{N}}(\mathbb{Z}/2\mathbb{Z})$ with the Tychonov topology (each factor being discrete) we obtain a topological group which is completely metrizable and compact Hausdorff.
Let $e_i$ be the element of $\prod_{\mathbb{N}}(\mathbb{Z}/2\mathbb{Z})$ with $\operatorname{supp}(e_i)=\{i\}$.
Then any homomorphism $\varphi: \prod_{\mathbb{N}}(\mathbb{Z}/2\mathbb{Z})\rightarrow \mathbb{Z}/2\mathbb{Z}$ which
sends all elements $e_i$ to the nontrivial element of $\mathbb{Z}/2\mathbb{Z}$, does not have open kernel. Therefore there is
a homomorphism from $\prod_{\mathbb{N}}(\mathbb{Z}/2\mathbb{Z})$ to $G$ which does not have open kernel.

\end{remark}

%\begin{remark}\marginpar{\tiny Delete}
%Note that in Theorems~\ref{metricandlcH} and~\ref{MCG_1} the image of $\varphi$ can be infinite: Consider $id:G\rightarrow G$ with the discrete topology on $G$. However these theorems show that either $\varphi:H\rightarrow G$ is almost continuous,
%or $im(\varphi)$ is {\it relatively small}.
%\end{remark}

\begin{remark}
Theorem~\ref{metricandlcH} implies that if $G$ is an acylindrically hyperbolic group, then
$G$ cannot be given a nondiscrete completely metrizable or locally compact Hausdorff group topology.
In particular, $\HEG$ cannot be given a nondiscrete topology of this kind.
\end{remark}

%\medskip

An exact computation of the elliptic radical is, in general, a difficult problem.
However, one can compute this radical in some important cases which we list below.
Recall that an acylindrical action of a group $G$ on a hyperbolic space $S$ is called {\it universal acylindrical} if the set of elliptic elements for this action coincides with the elliptic radical $\Rad_{\Ell}(G)$.
%Let $I$ be the set of all acylindrical actions of $G$ on hyperbolic spaces.
%Recall that $G$ admits a {\it universal acylindrical action on a hyperbolic space} if
%$\Rad_{\Ell}(G)=\Ell(G\curvearrowright S)$ for some action $(G\curvearrowright S)\in I$.
The universal (cobounded) acylindrical actions of groups on hyperbolic spaces were first introduced by Osin in~\cite{Osin_1} (using another terminology, see Section~6) and studied in~\cite{Abbott,ABO,ABD}. Many interesting groups admit such an action.
Among them are the following.

\begin{enumerate}
\item[(a)] Hyperbolic groups.

\item[(b)] Relatively hyperbolic groups with respect to a collection of subgroups $\{H_{\lambda}\}_{\lambda\in \Lambda}$ such that none of the $H_{\lambda}$'s is virtually cyclic or acylindrically hyperbolic.

\item[(c)] Mapping class groups of compact surfaces.

\item[(d)] Right-angled Artin groups.

\item[(e)] Fundamental groups of compact orientable 3-manifolds with empty or toroidal boundary.

\item[(f)] Hierarchically hyperbolic groups, see~\cite[Theorem A]{ABD}. This includes groups that act properly and cocompactly on a proper $CAT(0)$-complex, in particular, right-angled Artin groups and right-angled Coxeter groups.
\end{enumerate}

\medskip

\noindent
Cases (a)-(e) were considered in~\cite[Example 6.5]{Osin_1} and in~\cite[Theorem 2.18]{ABO}. Case~(f)
was considered in~\cite[Theorem A]{ABD}. The history and the sequence of appearance of these results are illuminated in~\cite{ABO,ABD}.

%relatively hyperbolic groups with respect
%to a collection of subgroups $\{H_{\lambda}\}_{\lambda\in \Lambda}$ and none of the $H_{\lambda}$'s
%is virtually cyclic or acylindrically hyperbolic,
%then $G$ acts universal acylindrically on the relative Cayley graph $\Gamma(G,X\cup \mathcal{H})$.

%\end{remark}

\medskip

%The elliptic radical of the mapping class group of an orientable surface
%In some cases infinite finite index subgroups of $\Rad_{\Ell}(G)$ have smaller elliptical radical.
%This is the case for the mapping class group. In this way we prove the following theorem.

Approaching to Theorem~\ref{MCG_1}, we focus on the case (c).
Let $\Sigma_{g,b}$ be an orientable surface of genus $g\geqslant 0$ with $b\geqslant 0$ boundary components.
In Subsection 9.1, we recall definitions of the mapping class group $\MCG(\Sigma_{g,b})$ and of
the curve graph ${\rm C}(\Sigma_{g,b})$.
It is known that if $3g+b-4>0$ then the curve graph $\C(\Sigma_{g,b})$ is hyperbolic (Masur-Minsky~\cite{MM}),
the action of $\MCG(\Sigma_{g,b})$ on  $\C(\Sigma_{g,b})$ is acylindrical (Bowditch~\cite{Bowditch}), and that the elliptic elements of this action are precisely elements of finite order and reducible elements (Masur-Minsky~\cite{MM}).

In this case Theorem~\ref{metricandlcH} says that for any homomorphism
$\varphi:H\rightarrow \MCG(\Sigma_{g,b})$, where $H$ as above, there is an open normal subgroup $V$ of $H$ such that $\varphi(V)$ consists of elements of finite order and reducible elements.
Theorem~\ref{MCG_1} strengthens this statement.

\begin{bigtheorem}\label{MCG_1}
Let $H$ be a topological group which is either completely metrizable or locally countably compact Hausdorff.
Let $\Sigma$ be a connected compact surface (possibly non-orientable and possibly with boundary components).
Then for any homomorphism $\varphi:H\rightarrow \MCG(\Sigma)$ there exists an open normal subgroup $V\leqslant H$ such that $\varphi(V)$ is finite.
\end{bigtheorem}

\medskip

The proof of this theorem uses a generalization of Theorem~\ref{metricandlcH} (see Theorem~\ref{generalisation_of_B}) and a theorem of Ivanov (see~\cite[Theorem 1]{Ivanov_2}) on the structure of subgroups of the mapping class groups.
\medskip

\begin{remark}
A homomorphism between topological groups $\varphi: H\rightarrow G$ is called {\it virtually continuous}
if there exists a finite index subgroup $H_1$ of $H$ such that the restriction $\varphi\upharpoonright H_1$ is continuous.
Using the fact that $\MCG(\Sigma)$ has a finite index torsion-free subgroup (see~\cite[Theorem 3]{Ivanov_2}), one can prove
that the statement of Theorem~\ref{MCG_1} is equivalent to the assertion that any homomorphism $\varphi:H\rightarrow \MCG(\Sigma)$
is virtually continuous, where $\MCG(\Sigma)$ is considered as a discrete topological group.

%2) Let $H^{\circ}$ be the maximal connected component of $H$ containing $1_H$.
%Wrong:Every open subgroup contains $H^{0}$.
%Since $H^{\circ}$ is closed, we can apply Theorem B to $H^{\circ}$ to see that $\varphi(H^{\circ})$ is finite.

\medskip

%{\it Proof.} Wrong: $H^{\circ}$ is connected in the induced topology. Then every open subgroup of $H^{\circ}$
%must coincide with $H^{\circ}$. This also proves that $V$ contains $H^{\circ}$.

\end{remark}

%\medskip

Note that $\MCG^{\pm}(\Sigma)=\Out(\pi_1(\Sigma))$ if $\Sigma$ is orientable and closed
(here $\MCG^{\pm}(\Sigma)$ is the full mapping class group of $\Sigma$).
%In general, $\MCG(\Sigma)$ is a subgroup of $Out(\pi_1(\Sigma))$.
Therefore it is natural to ask, whether Theorem~\ref{MCG_1} holds for $\varphi:H\rightarrow \Out(G)$,
where $G$ is a one-ended hyperbolic group.
Combining Theorem~C with a result of Levitt~\cite{Levitt} about the structure of $\Out(G)$, we deduce the following theorem.

%In~\cite[Section 3]{Fujiwara}, Fujiwarja noticed that the JSJ-decomposition of such $G$
%(described by Bowditch in~\cite{Bowditch_1}) implies that $Out(G)$ has a subgroup of finite index which is
%a direct product of a finitely generated free abelian group and finitely many mapping class groups
%of 2-orbifolds of finite volume. Moreover, Fujiwara proved that the mapping class group of such an orbifold embeds
%to the mapping class of a compact surface. Combining this with Theorem C we deduce the following theorem.

\begin{bigtheorem}\label{Out}
Let $H$ be a topological group which is either completely metrizable or locally countably compact Hausdorff.
Let $G$ be a one-ended hyperbolic group.
Then for any homomorphism $\varphi:H\rightarrow \Out(G)$
%almost factors through the canonical projection $H\rightarrow H/H_0$.
there exists an open normal subgroup $V\leqslant H$ such that $\varphi(V)$ is finite.
\end{bigtheorem}

%\noindent
%{\bf Problem 1.} {\it Does the analogous statement hold for ${\GL}_n(\mathbb{Z})$ and for ${\Out}(F_n)$?}

\medskip

\begin{remark}
It is interesting to compare Theorems~C and~D with the following results of Farb and Masur, and Fujiwara.

In~\cite[Theorem 1.1]{Farb_Masur} Farb and Masur proved the following.
Let $H$ be a real semisimple Lie group of rank at least 2 with finite center and no nontrivial compact factors, let $\Gamma$ be an irreducible lattice in $H$, and $S$ be a compact connected orientable surface.
Then any homomorphism $\varphi:\Gamma\rightarrow \MCG(S)$ has finite image.
Fujiwara proved the same for nonorientable surfaces and then
extended this result to homomorphisms $\varphi:\Gamma\rightarrow \Out(G)$, where $G$ is a one-ended hyperbolic group~\cite[Theorems 1 and 3]{Fujiwara}.
\end{remark}

Theorems~\ref{metricandlcH} and~\ref{thebigone} have many other applications.
As a corollary of Theorem~\ref{thebigone}, we show that $\HEG$ is an acylindrically hyperbolic group which does not admit any universal acylindrical action (see Proposition~\ref{HEG_No_Univer}).
This seems to be the first torsion-free example of this kind.
In \cite{Abbott} Abbott proved that Dunwoody's inaccessible finitely generated group does not have universal
acylindrical actions. Her proof extensively uses the fact that this group has unbounded torsion.

\medskip

\noindent
{\bf Problem 1.} {\it Does there exist a finitely generated (countable) torsion-free acylindrically hyperbolic group which does not admit universal acylindrical actions?}

\medskip

Note that since free groups admit universal acylindrical actions and $\HEG$ does not, this gives still another proof that
$\HEG$ is not a free group.
As another corollary of Theorem A, we recover a result of Eda from~\cite{Ed1} that
any endomorphism of the Hawaiian earring group is induced, up to conjugacy, by a continuous map
from the Hawaiian earring to itself preserving the basepoint $(0,0)$ (see Corollary~\ref{App1}).

Other applications include automatic continuity results for homomorphisms whose domain is as in
Theorem~\ref{metricandlcH} or~\ref{thebigone} and whose codomain is relatively hyperbolic or the fundamental group of a graph of groups.
To present some of them in a short form, we recall the following terminology.

A group $G$ is \emph{cm-slender} if every abstract homomorphism from a completely metrizable topological group $H$ to $G$ has open kernel \cite{CC}.  Similarly one defines $G$ to be \emph{lcH-slender} (respectively \emph{n-slender}) by replacing $H$ with a locally compact Hausdorff group (respectively $\HEG$).
The abbreviation \emph{lccH-slender} stands for locally countably compact slenderness.

Free (abelian) groups were classically shown to be cm-slender, lcH-slender and n-slender (see \cite{Du} and \cite{Hi}).  More recent work has shown that torsion-free word hyperbolic groups, right-angled Artin groups, braid groups and many other groups satisfy various of these slenderness conditions (see \cite{N}, \cite{CC}, \cite{CK}, \cite{KV}). Note that a group which is either n-slender, cm-slender, or lccH-slender must be torsion-free.
We prove the following combination theorems.

\begin{bigtheorem}\label{slendernessrelhyp}
If $G$ is torsion-free and relatively hyperbolic with respect to
a collection $\{H_{\lambda}\}_{\lambda \in \Lambda}$ of n-slender (respectively cm-slender, lcH-slender, lccH-slender) subgroups, then $G$ is also n-slender (respectively cm-slender, lcH-slenderm, lccH-slender).
\end{bigtheorem}

\begin{bigtheorem}\label{slendernessgraph}
Suppose that $G$ is the fundamental group of a (possibly infinite) graph of groups $(\Gamma,\mathbb{G})$ where each vertex group $G_v$ is n-slender (respectively cm-slender, lcH-slender, lccH-slender) and that the action of $G$ on the Bass-Serre tree associated with $(\Gamma,\mathbb{G})$ is $(k,n)$-acylindrical for some $k,n\in \mathbb{N}$.  Then $G$ is n-slender (respectively cm-slender, lcH-slender, lccH-slender).
\end{bigtheorem}

\begin{remark}
%This will be a short historical remark about automatic (almost) continuity in different categories of groups.

%1) Recall that one of equivalent definitions of a profinite group says that a group $G$ is profinite
%if it is a compact Hausdorff and totally disconnected topological group. A topological
%group is said to be {\it finitely generated in topological sense} if it has a finitely generated dense subgroup.

%In~\cite[Theorem 1.1]{NS} Nikolov and Segal proved that every homomorphism from a finitely generated in topological
%sense profinite group to any profinite group is continuous.

In~\cite[Theorem D]{KV} Kramer and Varghese proved that given an abstract homomorphism $\varphi$ from a locally compact Hausdorff group $H$ to a group $G$ in which torsion subgroups are finite and abelian subgroups are a direct sum of cyclic groups, it is either the case that $\ker(\varphi)$ is open or that $\im(\varphi)$ lies in the normalizer of a finite non-trivial subgroup of $G$.

Note that the technique they use cannot be applied in the case of completely metrizable or locally countably
compact Hausdorff groups.
\end{remark}

The paper is organized as follows.  We first provide some background in Sections~\ref{HE} and~\ref{AH}.
In Section~\ref{ProofofTheoremA} we give a proof of Theorem~\ref{thebigone}, which serves as
a basis for the proof of Theorem~\ref{metricandlcH}.
A weak version of Theorem~\ref{metricandlcH} is proved in Section~\ref{ProofofTheoremBweak}.
Universally acylindrical actions and the elliptic radical are discussed in Section~\ref{Universal}.
Theorem~\ref{metricandlcH} and its strong version are proved in Sections~\ref{ProofofTheoremB} and~\ref{ProofofTheoremBstrong}, respectively.
Theorems~\ref{MCG_1} and~\ref{Out} are proved Section~\ref{MCG}. Some general applications are given in Section~\ref{GeneralApplications}.
Theorems~\ref{slendernessrelhyp} and~\ref{slendernessgraph} are proved in Section~\ref{Automatic_Continuity}.

We thank Andreas Thom for a useful discussion about a result of Koubi from~\cite{Koubi}.

\end{section}

\begin{section}{Hawaiian earring and its fundamental group}\label{HE}

We define $\mathbb{N}=\{1,2,\dots\}$ and $\mathbb{N}_0=\mathbb{N}\cup \{0\}$.
The {\it Hawaiian earring} is the subspace $\operatorname{E}=\bigcup_{n\in \mathbb{N}} C_n$ of the Euclidean plane $\mathbb{R}^2$ where $C_n$ is the circle with center $(-\frac{1}{n},0)$ and radius $\frac{1}{n}$, see Fig.~1.

\vspace*{-30mm}
\hspace*{35mm}
\includegraphics[scale=0.7]{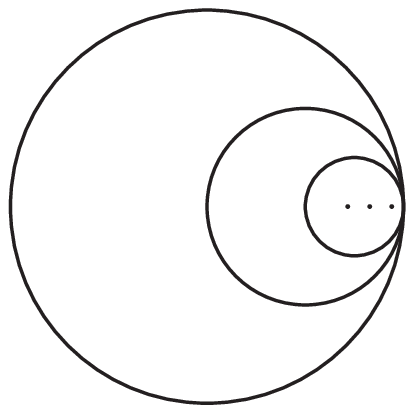}

\vspace*{-14cm}

\begin{center}
Fig. 1. Hawaiian earring
\end{center}

\medskip

We stress that $\operatorname{E}$ is a topological space with the topology induced by that of the Euclidean plane.
The fundamental group of $\operatorname{E}$ with respect to the basepoint (0,0) is denoted by $\HEG$ in the present paper.  The unusual properties of this group have been explored over the last several decades.  Higman showed that $\HEG$ is not free (see \cite{Hi} and \cite{Griffiths}).  Much later the abelianization of $\HEG$ was explicitly computed by Eda and Kawamura \cite{EdKaw} and was shown to include exotic direct summands, as well as the group $\mathbb{R}$.

We recall a combinatorial description of this group, which somewhat resembles that for free groups (see~\cite{Ed1}).

For a totally ordered set $X$ let $X^{\ast}$ be the same set with the reverse order.
For two totally ordered sets $X$ and $Y$ we define their concatenation as the
set $X\sqcup Y$ together with the order which preserves that of both $X$ and $Y$ and places elements in $X$ below those of $Y$.

Let $A=\{a_n^{\pm 1}\}_{n\in \mathbb{N}}$ be a countably infinite set with formal inverses.
By $W(A)$ we denote the class of all maps $w:X\rightarrow A$ with the following two properties:

(1) $X$ is a totally ordered set;

(2) for each $a\in A$ the set $\{i\in X\,|\, w(i)=a\}$ is finite.

\medskip

The elements of $W(A)$ are called {\it words}.
For a word $w:X\rightarrow A$ we define the {\it inverse word} $w^{-1}:X^{\ast}\rightarrow A$ by setting
$w^{-1}(i) = (w(i))^{-1}$.
For two words $w:X\rightarrow A$ and $u:Y\rightarrow A$ we define their concatenation $wu:X\sqcup Y\rightarrow A$
by
$$wu(i) = \begin{cases}w(i) \text{ if } i\in X\\ u(i) \text{ if }i\in Y\end{cases}.$$

\noindent Note that $W(A)$ is not a group since $W(A)$ is a class (not a set yet) and since the concatenation $ww^{-1}$ does not coincide with the trivial map $\emptyset:\emptyset \rightarrow A$.

To define a group, we first define an equivalence relation $\equiv$ on $W(A)$:
Given two words $w:X\rightarrow A$ and $u:Y\rightarrow A$, we write
$w \equiv u$ provided there exists an order isomorphism $\varphi: X \rightarrow Y$ such that $w(i) = u(\varphi(i))$.
For brevity, we denote the set $W(A)/\equiv$ again by $W(A)$ and we will term the $\equiv$ classes as words.
Now we may assume that $W(A)$ is a set.

There are two ways to transform $W(A)$ to a group. The first way utilizes an appropriate definition of (infinite)
cancellation in words. The second way (which we describe below) uses restriction maps $p_n:W(A)\rightarrow F_n$,
where $n\in \mathbb{N}_0$ and $F_n$ is the free group with basis $\{a_1,\dots,a_n\}$.

For a word $w:X\rightarrow A$ we define
$$p_n(w) = w\upharpoonright\{i\in X\mid w(i)\in \{a_1^{\pm 1},\dots,a_n^{\pm 1}\}\}.$$
For $w, u \in W(A)$ we write $w \sim u$ provided that for every $n\in \mathbb{N}$ the words $p_n(w)$ and $p_n(u)$ are equal as elements in the free group $F_n$.
The quotient set $\HEG = W(A)/\sim$ possesses a group structure given by concatenation: $[w][u] = [wu]$ and $[u]^{-1} = [u^{-1}]$.

For each $n\in \mathbb{N}_0$, we define subgroups $\HEG_n$ and $\HEG^n$ of the group $\HEG$:

\begin{enumerate}
\item[{\small $\bullet$}] The subgroup $\HEG_n$ consists of those elements of $\HEG$ which have
a representative $w:X\rightarrow A$ satisfying $w(X) \subseteq \{a_1^{\pm 1},\dots ,a_n^{\pm 1}\}$.

\medskip

\item[{\small $\bullet$}] The subgroup $\HEG^n$  consists of those elements of $\HEG$ which have
a representative $w:X\rightarrow A$ satisfying $w(X) \cap \{a_1^{\pm 1},\dots ,a_n^{\pm 1}\} = \emptyset$.
\end{enumerate}

\noindent Then, for each $n\in \mathbb{N}_0$, there is a natural decomposition $\HEG \simeq \HEG_n \ast \HEG^n$.
Moreover, there are natural isomorphisms $\HEG_n\simeq F_n$ and $\HEG^n\simeq \HEG$.

\medskip

The group $\HEG$ can be considered as a topological group with respect to the topology $\mathcal{T}$ for which the set $\{\ker(p_n)\,|\, n\in \mathbb{N}_0\}$ is
a basis of open sets of neighborhoods of the neutral element. Note that $\ker(p_n)$ coincides with
the normal closure of the subgroup $\HEG^n$ in $\HEG$.

\section{Acylindric actions of groups on hyperbolic spaces}\label{AH}

In this paper, all actions of groups on metric spaces  are supposed to be isometric. Recall that a group $G$ acts {\it coboundedly} on a metric space $S$ if there exists a bounded subset $B\subset S$ such that
$S=G\cdot B$.

We work only with rectifiable paths in metric spaces. In most cases the paths under consideration are just finite concatenations of geodesics. At some point we consider a bi-infinite path, whose construction is given precisely.
For a path $p$, let $p_{-}$ and $p_{+}$ denote the initial and the terminal points of $p$ and $\ell(p)$ the length of $p$.
For $\varkappa\geqslant 1$ and $\varepsilon\geqslant 0$, we call the path $p$ a
{\it $(\varkappa,\varepsilon)$-quasi-geodesic} if for any subpath $q$ of $p$ holds
$$
d(q_{-},q_{+})\geqslant \frac{1}{\varkappa}\ell(q)-\varepsilon.
$$
Let $M>0$. A path $p$ is called an {\it $M$-local $(\varkappa,\varepsilon)$-quasi-geodesic} if for any subpath $q$ of $p$ with $\ell(q)\leqslant M$ the above inequality holds.

For two subsets $A,B$ in a metric space $S$ let $d_{\operatorname{Hau}}(A,B)$ be the {\it Hausdorff distance} between
$A$ and $B$.

\subsection{Hyperbolic spaces}

Let $A,B,C$ be three points in a metric space $S$.
Recall that the {\it Gromov product} of $A,B$ with respect to $C$ is the number
$$
(A,B)_C:=\frac{d(C,A)+d(C,B)-d(A,B)}{2}.
$$

We use the following definition of a $\delta$-hyperbolic space (see~\cite[Chapter III.H, Definition~1.16 and Proposition~1.17]{BH}).

For $\delta\geqslant 0$, we say that a geodesic triangle $ABC$ in $S$ is {\it $\delta$-thin at the vertex $C$}
if for any two points $A_1$ and $B_1$ on the sides $[C,A]$ and $[C,B]$ with $d(C,A_1)=d(C,B_1)\leqslant (A,B)_C$,
we have $d(A_1,B_1)\leqslant \delta$.  We say that a metric space $S$ is {\it $\delta$-hyperbolic} if it is geodesic and every geodesic triangle in $S$ is $\delta$-thin at each of its vertices.

\begin{lemma}\label{HausdorffQuasiGeod}
{\rm (see~\cite[Chapitre 3, Th$\acute{\text{\rm e}}$or$\grave{\text{\rm e}}$me 3.1]{CDP})}
For all $\delta\geqslant 0$, $\varkappa\geqslant 1$, $\varepsilon\geqslant 0$, there exists a constant
$\mu=\mu(\delta,\varkappa,\varepsilon)> 0$ with the following property:

If $S$ is a $\delta$-hyperbolic space and $p$ and $q$ are infinite $(\varkappa,\varepsilon)$-quasi-geodesics in $S$ with the same limit points on the Gromov boundary $\partial S$, then the Hausdorff distance between $p$ and $q$
is at most $\mu(\delta,\varkappa,\varepsilon)$.
\end{lemma}

\begin{lemma}\label{QuasiGeodLocalGlobal} {\rm (see~\cite[Chapitre 3,
Th$\acute{\text{\rm e}}$or$\grave{\text{\rm e}}$me 1.4]{CDP})}
Let $S$ be a $\delta$-hyperbolic space.
Let $k\geqslant 1$, $k'>k$ and $l\geqslant 0$. There exists a real number $M=M(\delta,k,k',l)>0$ such that
every $M$-local $(k,l)$-quasi-geodesic is a global $(k',l)$-quasi-geodesic.
\end{lemma}

\medskip

\subsection{Acylindrical actions}

\begin{definition}\label{def Bowditch} (see \cite{Bowditch}) An action $G\curvearrowright S$ of a group $G$ on a metric space $(S, d)$ is \textit{acylindrical} if for each $\epsilon>0$ there exist $R, N>0$ such that for any two points $p, q\in S$ satisfying $d(p, q) \geqslant R$ the set

\begin{center}  $\{g\in G \mid d(p, gp)\leqslant \epsilon\text{ and }d(q, gq) \leqslant \epsilon\}$

\end{center}

\noindent is of cardinality at most $N$.
\end{definition}

Recall that an action of a group $G$ on a hyperbolic space $S$ is called {\it elementary} if the limit set of $G$
on the Gromov boundary $\partial G$ contains at most 2 points.

\begin{definition} {\rm (see~\cite[Definition 1.3]{Osin_1})
A group $G$ is called {\it acylindrically hyperbolic} if it satisfies one of the following equivalent
conditions:

\begin{enumerate}
\item[{\rm (AH$_1$)}] There exists a generating set $X$ of $G$ such that the corresponding Cayley graph $\Gamma(G,X)$
is hyperbolic, $|\partial \Gamma (G,X)|>2$, and the natural action of $G$ on $\Gamma(G,X)$ is acylindrical.

\medskip

\item[{\rm (AH$_2$)}] $G$ admits a non-elementary acylindrical action on a hyperbolic space.
\end{enumerate}
}
\end{definition}

\noindent In the case (AH$_1$), we also write that $G$ is {\it acylindrically hyperbolic with respect to $X$}.

\begin{definition}
{\rm
Given a group $G$ acting on a metric space $S$, an element $g\in G$ is called {\it elliptic}
if some (equivalently, any) orbit of $g$ is bounded, and {\it loxodromic} if the map
$\mathbb{Z}\rightarrow S$ defined by
$n\mapsto g^nx$ is a quasi-isometric embedding for some (equivalently, any) $x\in S$. That is,
for $x\in S$, there exist $\varkappa\geqslant 1$ and $\varepsilon\geqslant 0$ such that for any $n,m\in \mathbb{Z}$ we have
$$
d(g^nx,g^mx)\geqslant \frac{1}{\varkappa} |n-m|-\varepsilon.
$$

\noindent The set of loxodromic (respectively elliptic) elements of $G$ with respect to this action
is denoted by $\Lox(G \curvearrowright S)$ (respectively by $\Ell(G \curvearrowright S)$).

}
\end{definition}

Bowditch~\cite[Lemma 2.2]{Bowditch} proved that every element of a group $G$ acting acylindrically on a hyperbolic space is either elliptic or loxodromic (see a more general statement in~\cite[Theorem 1.1]{Osin_1}).
Every loxodromic element $g$ of such group $G$
is contained in a unique maximal virtually cyclic subgroup~\cite[Lemma 6.5]{DGO}. This subgroup, denoted by $E_G(g)$, is called the {\it elementary subgroup associated with $g$}; it can be described as follows (see equivalent definitions in~\cite[Corollary~6.6]{DGO}):
$$
E_G(g)\! \!\! =\{f\in G\,|\, \exists  n\in \mathbb{N}:  f^{-1}g^nf=g^{\pm n}\}.
$$

\begin{lemma}\label{elem_index} {\rm (see~\cite[Lemma 6.8]{Osin_1})}
Suppose that a group $G$ acts acylindrically on a hyperbolic space $S$. Then there exists $L\in \mathbb{N}$
such that for every loxodromic element $g\in G$, $E_G(g)$ contains a cyclic subgroup
of index at most~$L$.
\end{lemma}

\begin{definition}\label{pseodolength} Suppose that $G$ acts on a metric space $(S,d)$. Let $s_0$ be a point in $S$.
Then we define a $G$-equivariant pseudo-metric $d_{s_0}$ on $G$ by the formula $d_{s_0}(g,h)=d(gs_0,hs_0)$, where $g,h\in G$.
We set $|g|_{s_0}=d_{s_0}(1,g)$. We call the function $|\,\cdot \,|_{s_0}:G\rightarrow \mathbb{R}_{\geqslant 0}$ the {\it length function} on $G$ induced by the action of $G$ on the based space $(S,s_0)$.
\end{definition}

Note that if $X$ is a generating set of $G$, then the length function on $G$
induced by the action of $G$ on the based Cayley graph $(\Gamma(G,X),1)$ is
the usual length function on $G$ with respect to $X$.

\begin{lemma}\label{Elliptic_cobounded}
Let $G$ be a group which acts coboundedly and acylindrically on a hyperbolic space S.
Let $s_0$ be a point in $S$ and $|\cdot |_{s_0}:G\rightarrow \mathbb{R}_{\geqslant 0}$ be the length function on $G$
induced by the action of $G$ on the based space $(S,s_0)$. Then there exists a constant $K> 0$ such that the following holds:  For any $g\in \Ell (G\curvearrowright S)$ there exists $h\in G$ such that $|h^{-1}g^nh|_{s_0}\leqslant K$ for any $n\in \mathbb{Z}$.
\end{lemma}

\medskip

\begin{proof} For any point $s\in S$ and any real $r> 0$ let $B_r(s)$ be the open ball in $S$ of radius $r$ with the center at $s$. Since the action of $G$ on $S$ is cobounded, there exists a constant $L> 0$ such that
$$
S=G\cdot B_L(s_0).\eqno{(3.1)}
$$
Let $\delta\geqslant 0$ be a constant such that $S$
is $\delta$-hyperbolic and let $g$ be an element from $\Ell (G\curvearrowright S)$.
Then there exists a point $s\in S$ such that the diameter of the orbit $\langle g\rangle s$ is less than $4\delta+1$
(see~\cite[Corollary~6.7]{Osin_1}; the proof can be also extracted from~\cite{BG}).
Thus, $\langle g\rangle s\subseteq B_{4\delta+1}(s)$. By (3.1), there exist an element $h\in G$ and a point $s_1\in B_L(s_0)$
such that $s=hs_1$. Then $h^{-1}\langle g\rangle h s_1\subseteq B_{4\delta+1}(s_1)$. Thus
$h^{-1}\langle g\rangle h s_0\subseteq B_K(s_0)$ for $K=2L+4\delta+1$.
\end{proof}

\medskip

In the following we use variants of~Lemma~2.10 and~Lemma 2.11 from~\cite{Bog_1}. These lemmas were formulated
for groups $G$ and generating sets $X$ of $G$ such that the Cayley graph $\Gamma(G,X)$ is hyperbolic and
$G$ acts acylindrically on $\Gamma(G,X)$.
They can be easily adapted to the case where $G$ acts coboundedly and acylindrically on a hyperbolic space.

\begin{lemma}\label{lem 2.2}
{\rm (see~\cite[Lemma 2.10]{Bog_1})}
Let $G$ be a group which acts coboundedly and acylindrically on a hyperbolic space $S$.
Let $s_0$ be a point in $S$ and $|\cdot |_{s_0}:G\rightarrow \mathbb{R}_{\geqslant 0}$ be the length function on $G$
induced by the action of $G$ on the based space $(S,s_0)$.
Then there exists a constant $N_0\in \mathbb{N}$ such that the following holds:
For any $a,b\in \Lox(G\curvearrowright S)$ with $E_G(a)\neq E_G(b)$ and for any $n,m\in \mathbb{N}$ we have that
$$|a^nb^m|_{s_0}>\frac{\min\{n,m\}}{N_0}.$$
\end{lemma}

\begin{lemma}\label{lem 2.3}
{\rm (see~\cite[Lemma 2.11]{Bog_1})}
Let $G$ be a group which acts coboundedly and acylindrically on a hyperbolic space $S$.
Let $s_0$ be a point in $S$ and $|\cdot |_{s_0}:G\rightarrow \mathbb{R}_{\geqslant 0}$ be the length function on $G$
induced by the action of $G$ on the based space $(S,s_0)$.
Then there exists a constant $N_1\in \mathbb{N}$ such that the following holds:
For any $a\in \Lox(G\curvearrowright X)$, for any $b\in \Ell(G\curvearrowright X)\setminus E_G(a)$, and for any $n\in \mathbb{N}$ we have that
$$|a^nb|_{s_0}>\frac{n}{N_1}.$$
\end{lemma}

\begin{corollary}\label{lox-lox}
Let $G$ be a group which acts coboundedly and acylindrically on a hyperbolic space S.
Then there exists a constant $C_0\in \mathbb{N}$ such that the following holds:
For any $a,b\in \Lox(G\curvearrowright S)$ with $E_G(a)\neq E_G(b)$ and for any $n,m\geqslant C_0$ we have that $a^{n}b^m\in \Lox(G\curvearrowright S)$.
\end{corollary}

\begin{proof}
Let $s_0$ be a point in $S$ and $|\cdot |_{s_0}:G\rightarrow \mathbb{R}_{\geqslant 0}$ be the length function on $G$
induced by the action of $G$ on the based space $(S,s_0)$.
We set $C_0=\lceil K \rceil N_0$, where $K$ and $N_0$ are the constants from Lemmas~\ref{Elliptic_cobounded} and~\ref{lem 2.2},
respectively.

Suppose to the contrary that there exist $a,b\in \Lox(G\curvearrowright S)$ and $n,m\geqslant C_0$ such that $a^nb^m\in \Ell(G\curvearrowright S)$.
By Lemma~\ref{Elliptic_cobounded}, there exists $g\in G$ such that $|g^{-1}(a^nb^m)g|_{s_0}\leqslant K$. We set $a_1=g^{-1}ag$, $b_1=g^{-1}bg$. Then $a_1,b_1\in  \Lox(G\curvearrowright S)$ and $|a_1^nb_1^m|_{s_0}\leqslant K$.
This contradicts Lemma~\ref{lem 2.2}, which says that $$|a_1^nb_1^m|_{s_0}>\frac{\min\{n,m\}}{N_0}\geqslant K.$$
\end{proof}

\medskip

The proof of the following corollary follows analogously from Lemma~\ref{lem 2.3}.

\begin{corollary}\label{lox-ell}
Let $G$ be a group which acts coboundedly and acylindrically on a hyperbolic space $S$.
There exists a constant $C_1\in \mathbb{N}$ such that the following holds:

If $a\in \Lox(G\curvearrowright S)$ and $b\in \Ell(G\curvearrowright S)\setminus E_G(a)$, then $a^nb\in \Lox(G\curvearrowright S)$ for any $n\geqslant C_1$.
\end{corollary}

\medskip

\subsection{Translation lengths and stable norms.}

The aim of this subsection is to prove Corollary~\ref{length_power_1} that will be used in the proof of Theorem~A.

\begin{definition} (see~\cite{CDP} and~\cite{Gromov})
Suppose that $G$ is a group acting on a metric space~$S$. Let $g$ be an element of $G$.

The {\it translation length} of $g$ is the number $[g]$ defined by
$$
[g]=\underset{s\in S}{\inf} d(s,gs).
$$

The {\it stable norm} of $g$ is the number $||g||$ defined by
$$
||g||=\underset{n\rightarrow \infty}{\lim}\frac{1}{n} d(s,g^ns),
$$
where $s$ is an arbitrary point of $S$. One can show that this definition does not depend on the choice of~$s$.
This limit exists since the function $f:\mathbb{N}\rightarrow \mathbb{R}_{\geqslant 0}$, $n\mapsto d(s,g^ns)$, is subadditive.

\end{definition}

The following lemma is quite trivial.

\begin{lemma}\label{inf} The stable norm satisfies the following properties.

\begin{enumerate}
\item[{\rm (1)}] $\displaystyle{[g]\geqslant ||g||=\underset{n\in \mathbb{N}}{\inf}\frac{d(s,g^ns)}{n}}$.

\medskip

\item[{\rm (2)}] $||x^{-1}gx||=||g||$ for any $x\in G$.

\medskip

\item[{\rm (3)}] $||g^k||=|k|\cdot ||g||$ for any $k\in \mathbb{Z}$.
\end{enumerate}
\end{lemma}

\begin{lemma}\label{32delta} {\rm \cite[Chapitre 10, Proposition 6.4]{CDP}}
Let $S$ be a $\delta$-hyperbolic space and $g$ a loxodromic isometry of $S$.
Then
$$
||g||\leqslant [g]\leqslant ||g||+32\delta.
$$
\end{lemma}

\begin{definition}
Let $G$ be a group acting on a metric space $S$ and let $g$ be a loxodromic element in $G$ and $s$ be a point in $S$.
We choose a geodesic path $p_0$ with $(p_0)_{-}=s$ and $(p_0)_{+}=gs$ and set $p_i=g^ip_0$ for any $i\in \mathbb{Z}$.
The bi-infinite path $\dots p_{-1}p_0p_1\dots $ is denoted by $L(s,g)$.
For $\varepsilon\geqslant 0$ we call the path $L(s,g)$ an {\it $\varepsilon$-line} of~$g$ if $d(s,gs)\leqslant [g]+\varepsilon$.
We define an {\it $\varepsilon$-axis} of $g$ as follows:
$$
A_{\varepsilon}(g)=\{x\in S\,|\, d(x,gx)\leqslant [g]+\varepsilon\}.
$$
Clearly, any $\varepsilon$-line of $g$ lies in the $\varepsilon$-axis of $g$.

\begin{lemma}\label{LemmaAboutAxes} {\rm(see~\cite[Proposition 2.28]{Coulon})}
Let $g$ be a loxodromic isometry of a $\delta$-hyperbolic space $S$ and $s$ be an arbitrary point of $S$.
Then
$$
d(s,gs)\geqslant 2d(s,A_{8\delta}(g))+[g]-6\delta.
$$
\end{lemma}

\end{definition}

\begin{definition} Let $G$ be a group acting on a metric space.
The {\it injectivity radius} for the action of $G$ on $S$ is the number
$$
\mathbf{inj}\,(G\curvearrowright S)=\inf \{||g||: g\in \Lox(G\curvearrowright S)\}.
$$
\end{definition}

Bowditch proved the following important lemma.

\begin{lemma}\label{Bow}{\rm (\cite[Lemma 2.2]{Bowditch})}
If $G$ acts acylindrically on a hyperbolic space $S$, then $\mathbf{inj}\, (G\curvearrowright S)>0$.
\end{lemma}

We deduce from this lemma the following corollary.

\begin{corollary}\label{AxisQuasiGeod}
Let $G$ be a group which acts acylindrically on a $\delta$-hyperbolic space~$S$.
Then for any $\varepsilon\geqslant 0$, there exist constants $\varkappa=\varkappa(\varepsilon)\geqslant 1$ and
$\tau=\tau(\varepsilon)\geqslant 0$ such that
for any $g\in \Lox(G\curvearrowright S)$ we have that any $\varepsilon$-line of $g$ is a $(\varkappa,\tau)$-quasi-geodesic.
\end{corollary}

%\medskip

\begin{proof}
We set $M=M(\delta,1,2,\varepsilon+\delta)$, see Lemma~\ref{QuasiGeodLocalGlobal}.
Let $g$ be a loxodromic element of $G$ and let $L(s,g)$ be an $\varepsilon$-line of $g$.
We call the points $A_i=g^is$ the {\it phase points} of $L(s,g)$. Let $p_i$ be the subpath of $L(s,g)$ with $(p_i)_{-}=A_i$
and $(p_i)_{+}=A_{i+1}$. Clearly, $p_i$ is a geodesic.

{\it Case 1.}  Suppose that $[g]>M$.

We claim that $L(s,g)$ is an $M$-local $(1,\varepsilon + \delta)$-quasi-geodesic.
Indeed, let $q$ be a subpath of $L(s,g)$ of length less than or equal to $M$. If $q$ is contained in some $p_n$,
then $q$ is a geodesic and the claim holds. Suppose that $q$ is not contained in any $p_n$. Since $\ell(p_n)>M$ for any $n$ and
$\ell(q)\leqslant M$, the path $q$ decomposes as $q=q_1q_2$, where $q_1$ is a terminal subpath of some $p_{i-1}$ and
$q_2$ is an initial subpath of some $p_i$.
Let $q_3$ be a geodesic path connecting $q_{-}$ and $q_{+}$, see Fig. 2.

\vspace*{-30mm}
\hspace*{0mm}
\includegraphics[scale=0.7]{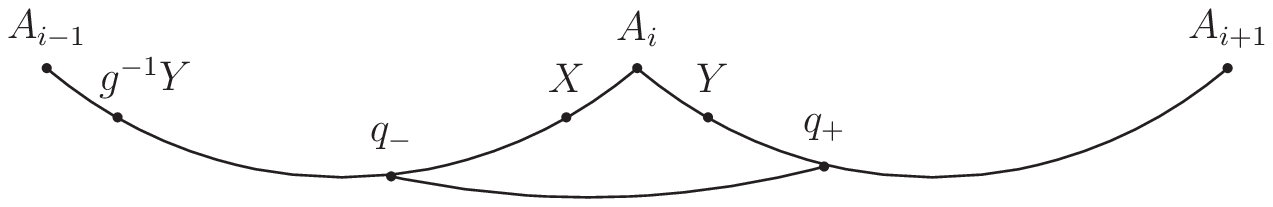}

\vspace*{-15.5cm}

\begin{center}
Fig. 2. A part of the $\varepsilon$-line $L(s,g)$.
\end{center}

Let $\Delta$ be the geodesic triangle with the sides $q_1,q_2,q_3$ and
let $X\in q_1$ and $Y\in q_2$ be the points such that $d(A_i,X)=d(A_i,Y)=(q_{-},q_{+})_{A_i}$.
Since $\Delta$ is $\delta$-thin, we have $d(X,Y)\leqslant \delta$.

{\it Subcase 1.} Suppose that $d(A_i,X)>\frac{\varepsilon+\delta}{2}$. Then
$$
\begin{array}{ll}
d(g^{-1}Y,Y) & \leqslant d(g^{-1}Y,X)+d(X,Y)\vspace*{2mm}\\
& = d(A_{i-1},A_i)-2d(X,A_i)+d(X,Y)\vspace*{2mm}\\
& < ([g]+\varepsilon )-(\varepsilon+\delta)+\delta=[g],
\end{array}
$$
a contradiction.

{\it Subcase 2.} Suppose that $d(A_i,X)\leqslant\frac{\varepsilon+\delta}{2}$.  Then
$$
d(q_{-},q_{+})=\ell(q_1)+\ell(q_2)-2(q_{-},q_{+})_{A_i}=
\ell(q)-2d(A_i,X)\geqslant \ell(q)-(\varepsilon+\delta).
$$

Therefore $L(s,g)$ is an $M$-local $(1,\varepsilon+\delta)$-quasi-geodesic.
By Lemma~\ref{QuasiGeodLocalGlobal}, $L(s,g)$ is a global $(2,\varepsilon+\delta)$-quasi-geodesic.

\medskip

{\it Case 2.} Suppose that $[g]\leqslant M$.

Denote $\theta={\text{\rm {\bf inj}}}\,(G\curvearrowright S)$. By Lemmas~\ref{inf} and~\ref{Bow} we have
$$
[g^n]\geqslant ||g^n||=n||g||\geqslant n\theta>0\eqno{(3.2)}
$$
for all natural $n$. In particular,
$$
M\geqslant [g]\geqslant \theta.\eqno{(3.3)}
$$
We prove that $L(s,g)$ is an $(\alpha,\beta)$-quasi-geodesic for $\alpha=\frac{3(M + \varepsilon)}{\theta}$
and $\beta = 2(M+\varepsilon)$.
Let $p$ be a subpath of $L(s,g)$. If $p$ does not contain phase points of $L(s,g)$,
then $p$ is a geodesic path, and hence a $(1,0)$-quasi-geodesic.
Suppose that $p$ contains at least one phase point of $L(s,g)$. Let $q$ be the maximal subpath of $p$ such that
$q_{-}=g^is$ and $q_{+}=g^{i+n}s$ for some $i\in \mathbb{Z}$ and $n\in \mathbb{N}_0$.   If $n = 0$ then $\ell(p) < 2([g] + \varepsilon)$  and it is easy to check that  $$d(p_{-}, p_{+}) \geqslant  0  > \frac{1}{\alpha}\ell(p) - \beta.$$  If $n>0$ then
$$
\begin{array}{ll}
d(p_{-},p_{+}) & \geqslant d(q_{-},q_{+})-2([g]+\varepsilon)\vspace*{2mm}\\
& \geqslant [g^n]-2([g]+\varepsilon)\vspace*{2mm}\\
& \geqslant n\theta-2([g]+\varepsilon)\vspace*{2mm}\\
& \geqslant \frac{1}{\alpha}(3n(M+\varepsilon)) - \beta\vspace*{2mm}\\
& \geqslant \frac{1}{\alpha}(n+2)([g] + \varepsilon) - \beta\vspace*{2mm}\\
& \geqslant \frac{1}{\alpha}\ell(p) - \beta.
\end{array}
$$
Thus, $L(s,g)$ is an $(\alpha,\beta)$-quasi-geodesic in this case.
In both cases $L(s,g)$ is a $(\varkappa,\tau)$-quasi-geodesic for $\varkappa=\max\{2,\alpha\}$
and $\tau=\max\{\varepsilon + \delta, \beta\}$.
\end{proof}

\begin{lemma}\label{HausdorffDistance}
Let $G$ be a group which acts acylindrically on a $\delta$-hyperbolic space~$S$.
For any $\varepsilon\geqslant 0$, there exists a constant $\nu=\nu(\delta,\varepsilon)>0$ such that for any loxodromic element $g\in G$
and any nonzero integers $n,m$, we have
$$
d_{\operatorname{Hau}} (A_{\varepsilon}(g^n),A_{\varepsilon}(g^m))\leqslant \nu(\delta,\varepsilon).
$$
\end{lemma}

\medskip

\begin{proof}
Let $s\in A_{\varepsilon}(g^n)$ and $t\in A_{\varepsilon}(g^m)$ be arbitrary points.
By Corollary~\ref{AxisQuasiGeod}, the bi-infinite paths $L(s,g^n)$ and $L(t,g^m)$ are $(\varkappa(\varepsilon),\tau(\varepsilon))$-quasi-geodesics.
Since these quasi-geodesics have the same
limit points on the Gromov boundary of $S$, then, by Lemma~\ref{HausdorffQuasiGeod}, there exists
a number $\mu=\mu(\delta,\varkappa(\varepsilon),\tau(\varepsilon))$ such that
$$
d_{\operatorname{Hau}}(L(s,g^n),L(t,g^m))\leqslant \mu.
$$
Since $s\in L(s,g^n)$ and $L(t,g^m)\subseteq A_{\varepsilon}(g^m)$,
we have
$$
d(s,A_{\varepsilon}(g^m))\leqslant \mu,
$$
and the proof is completed.
\end{proof}

\begin{proposition}\label{DifferenceOfLengths}
Let $G$ be a group which acts acylindrically on a $\delta$-hyperbolic space $S$.
There exists a constant $c_0>0$ with the following property:
Let $s_0$ be a point in $S$ and $|\cdot |_{s_0}:G\rightarrow \mathbb{R}_{\geqslant 0}$ be the length function on $G$ induced by the action of $G$ on the based space $(S,s_0)$.
Then for any loxodromic element $g\in G$ and any nonzero integers $n$ and $m$
we have
$$
|g^n|_{s_0}-|g^m|_{s_0}\geqslant [g^n]-[g^m]-c_0.
$$
In particular,
$$
\hspace*{15mm}|g^n|_{s_0}-|g^m|_{s_0}\geqslant (|n|-|m|)\,||g||-(c_0+32\delta).
$$
\end{proposition}

\medskip

\begin{proof}  %Recall that $A_{8\delta}(g^n)$ is the $8\delta$-axis of $g^n$.
Let $\lambda>0$ be an arbitrary real number. Then there exists a point $u\in A_{8\delta}(g^n)$ such that
$$d(s_0,u)<d(s_0,A_{8\delta}(g^n))+\lambda.$$
Using Lemma~\ref{LemmaAboutAxes}, we have
$$
\begin{array}{ll}
|g^n|_{s_0} & =d(s_0,g^ns_0)\geqslant 2d(s_0,A_{8\delta}(g^n))+[g^n]-6\delta\vspace*{2mm}\\
& > 2d(s_0,u)+[g^n]-(6\delta+2\lambda).
\end{array}\eqno{(3.4)}
$$

By Lemma~\ref{HausdorffDistance}, there exists  a point $v\in A_{8\delta}(g^m)$ such that $d(u,v)\leqslant \nu(\delta,8\delta)$.
Then
$$
\begin{array}{ll}
|g^m|_{s_0}=d(s_0,g^ms_0) &\leqslant d(s_0,v)+d(v,g^mv)+d(g^mv,g^ms_0)\vspace*{2mm}\\
& = 2d(s_0,v)+d(v,g^mv)\vspace*{2mm}\\
& \leqslant 2(d(s_0,u)+d(u,v))+d(v,g^mv)\vspace*{2mm}\\
& \leqslant 2\bigl(d(s_0,u)+\nu(\delta,8\delta)\bigr)+\bigl([g^m]+8\delta\bigr).
\end{array}\eqno{(3.5)}
$$
It follows from (3.4) and (3.5) that
$$
|g^n|_{s_0}-|g^m|_{s_0}>[g^n]-[g^m]-\bigl(14\delta+ 2\nu(\delta,8\delta)+2\lambda \bigr).
$$
Since $\lambda>0$ was arbitrary, we can set $c_0=14\delta+ 2\nu(\delta,8\delta)$
to satisfy the first inequality in the proposition.
The second inequality follows from the first one
with the help of Lemmas~\ref{inf} and~\ref{32delta}, which give the following estimations:
$$
\begin{array}{ll}
{[g^n]} & \geqslant ||g^n||=|n|\cdot ||g||,\vspace*{2mm}\\
{[g^m]} & \leqslant ||g^m||+32\delta=|m|\cdot ||g||+32\delta.
\end{array}
$$

\end{proof}

\medskip

Using the fact that $||g||\geqslant {\text{\rm{\bf inj}}}\,(G\curvearrowright S)>0$, we deduce from Proposition~\ref{DifferenceOfLengths} the following corollary.

\begin{corollary}\label{length_power_1}
Let $G$ be a group which acts acylindrically on a hyperbolic space~$S$.
There exist constants $\alpha\geqslant 1$ and $\beta\geqslant 0$ such that the following holds:
Let $s_0$ be a point in $S$ and $|\cdot |_{s_0}:G\rightarrow \mathbb{R}_{\geqslant 0}$ be the length function on $G$ induced by the action of $G$ on the based space $(S,s_0)$.
Then for any loxodromic element $g\in G$ and any natural number $n$ we have
$$
|g^n|_{s_0}\geqslant |g|_{s_0}+\frac{1}{\alpha}n-\beta.
$$
\end{corollary}

%\subsection{Elliptic subgroups of acylindrically hyperbolic groups.}

\end{section}

\begin{section}{Proof of Theorem~\ref{thebigone}}\label{ProofofTheoremA}

Let $s_0$ be a point in $S$ and $|\cdot |_{s_0}:G\rightarrow \mathbb{R}_{\geqslant 0}$ be the length function on $G$
induced by the action of $G$ on the based space $(S,s_0)$. To shorten the proof, we write $|g|$
instead of $|g|_{s_0}$, $\Lox(G)$ instead of $\Lox (G\curvearrowright S)$,
and $\Ell(G)$ instead of $\Ell (G\curvearrowright S)$.

Suppose the contrary. Then for any $n\in \mathbb{N}$, there exists $u_n\in \HEG^n$ such that $\varphi(u_n)\in \Lox(G)$.
We define $w_1\in \HEG$ inductively by the formula $w_i=u_i^nw_{i+1}^{m_i}$,
where $n$ and $m_i$ are appropriate numbers, which we choose below. Thus,
$$
w_1=u_1^n(u_2^n(\dots )^{m_2})^{m_1}.
$$
Denote $U_i=\varphi(u_i)$ and $W_i=\varphi(w_i)$. We have $W_i=U_i^nW_{i+1}^{m_i}$.
We have that $U_i\in \Lox(G)$, but we cannot assert that $W_i\in \Lox(G)$.

\medskip

$\bullet$ By Lemma~\ref{elem_index}, $E_G(U_i)$ contains a cyclic subgroup of index at most $L$.
Let $z_i$ be a generator of such a subgroup. Then there exists an integer $K_i$ such that
$$U_i^{L!}=z_i^{K_i}.\eqno{(4.1)}$$

Replacing $z_i$ by $z^{-1}_i$ if necessary, we may assume that $K_i\in \mathbb{N}$.
In the following definition of numbers $n$ and $m_i$ for $i\in \mathbb{N}$ we use

$\bullet$ the numbers $C_0$ and $C_1$ from Corollaries~\ref{lox-lox} and~\ref{lox-ell}, and

$\bullet$ the numbers $\alpha\geqslant 1$ and $\beta\geqslant 0$ from Corollary~\ref{length_power_1}.
Increasing $\alpha$ and $\beta$ if necessary, we may assume that $\alpha,\beta \in \mathbb{N}$.

\medskip

\noindent
We set
$N=\max\{C_0,C_1\}$ and
$$
n=NL!,\hspace*{5mm}
M_i=\Bigl\lceil \max\{ |U_i^n|,i+1\}\Bigr\rceil,
\hspace*{5mm}
m_i=2{\alpha}\beta nM_iK_i
\hspace{5mm}{\text{\rm for}}
\hspace*{2mm} i\in \mathbb{N}.\eqno{(4.2)}
$$

\medskip

{\bf Claim 1.} For all $i\in \mathbb{N}$, we have $W_i\in \Lox(G)$.

\medskip

\begin{proof} Recall that $W_i=U_i^nW_{i+1}^{m_i}$. We consider four cases.

{\it Case 1.}  Suppose that $W_{i+1}\in \Lox(G)$ and $E_G(U_i)\neq E_G(W_{i+1})$. Then $W_i\in \Lox(G)$ by Lemma~\ref{lox-lox}.

\medskip

{\it Case 2.}  Suppose that $W_{i+1}\in \Lox(G)$ and $E_G(U_i)=E_G(W_{i+1})$.
Then there exists an integer $P_i$ such that
$$
W_{i+1}^{L!}=z_i^{P_i}.\eqno{(4.3)}
$$
From (4.1) and (4.3), and using (4.2), we deduce
$$
W_i=U_i^nW_{i+1}^{m_i}=z_i^{NK_i+2\alpha\beta NM_iK_iP_i}.
$$
Hence $W_i$ is a nontrivial power
of a loxodromic element. Therefore $W_i\in \Lox(G)$.

\medskip

{\it Case 3.} Suppose that $W_{i+1}\in \Ell(G)$ and $W^{m_i}_{i+1}\notin E_G(U_i)$.
Then $W_i\in \Lox(G)$ by Corollary~\ref{lox-ell}.

\medskip

$\bullet$ Thus, if $W_{i+1}$ satisfies assumptions of one of the Cases 1-3, then $W_i\in \Lox(G)$.

\medskip

{\it Case 4.} Suppose that $W_{i+1}\in \Ell(G)$ and $W^{m_i}_{i+1}\in E_G(U_i)$.
It follows that

1) $W_i=U_i^nW_{i+1}^{m_i}\in E_G(U_i)$,

2) $W^{m_i}_{i+1}$ has a finite order dividing $L!$.

\medskip

$\bullet$ The element $W_{i+2}$ cannot satisfy assumptions of Cases 1-3, since $W_{i+1}\in \Ell(G)$.
Thus, $W_{i+2}$ satisfies assumptions of Case 4, and we have

\medskip

1') $W_{i+1}\in E_G(U_{i+1})$.

Since $W_{i+1}$ is an elliptic element in $E_G(U_{i+1})$, its order divides $L!$.
Since $L!$ is a divisor of $m_i$, we have $W_{i+1}^{m_i}=1$.
Then $W_i=U_i^n\in \Lox(G)$. The proof of the claim is completed.
\end{proof}

\medskip

By Corollary~\ref{length_power_1}, we have
$$
|W_{i+1}^{m_i}|\geqslant |W_{i+1}| +\frac{m_i}{\alpha} -\beta.
$$
From (4.2) we deduce $|U_i^n|\leqslant M_i\leqslant \frac{m_i}{2\alpha}$. Then
$$
|W_i|=|U_i^nW_{i+1}^{m_i}|\geqslant |W_{i+1}^{m_i}|-|U_i^n|\geqslant
|W_{i+1}|+\frac{m_i}{2\alpha}-\beta\geqslant |W_{i+1}|+i.
$$
It follows that
$$
|W_1|\geqslant |W_2|+1\geqslant |W_3|+2+1\geqslant \dots \geqslant |W_{i+1}|+i+(i-1)+\dots +1\geqslant \frac{i(i+1)}{2}
$$
for all $i\in \mathbb{N}$.
A contradiction.\hfill $\Box$

\section{Proof of a weak version of Theorem~\ref{metricandlcH}}\label{ProofofTheoremBweak}

\medskip

In this section we prove the following weak version of Theorem~\ref{metricandlcH}.

\begin{proposition}\label{weak_metricandlcH}
Let $H$ be a topological group which is either completely metrizable or locally countably compact Hausdorff,
$G$ be a group which acts coboundedly and acylindrically on a hyperbolic space $S$, and
$\varphi:H \rightarrow G$ be a group homomorphism.
Then there is an open neighborhood $V$
of identity $1_H$ such that $\varphi(V)\subseteq \Ell(G\curvearrowright S)$.
\end{proposition}

Preparing to the proof, we give the following definition, notation and define some constants.

\medskip

\begin{definition}
We call an element $h\in H$ {\it pre-loxodromic} if $\varphi(h)\in \Lox(G\curvearrowright S)$.
\end{definition}

\medskip

\begin{notation} For a subset $U\subseteq H$ and a natural number $k$ we denote
$$U^k:=\{h^k\,|\, h\in U\}, \hspace*{2mm} U^{1/k}:=\{h\in H\,|\, h^k\in U\}.$$
\end{notation}

\begin{constants} Select universal constants $L, C_0, C_1$ for the group $G$ from Lemma~\ref{elem_index}, Corollary~\ref{lox-lox}, and Corollary~\ref{lox-ell}.  Also select numbers $\alpha\geqslant 1$ and $\beta\geqslant 0$ from
Corollary~\ref{length_power_1}, without loss of generality both being natural numbers.  As in the proof of Theorem~~\ref{thebigone} we let $N = \max\{C_0, C_1\}$ and $$n = N L!.$$
\end{constants}

\medskip

{\it Proof of Proposition~\ref{weak_metricandlcH}}.
The proof is a modification of the proof of Theorem A.

{\bf (a)} Assume first that $H$ is completely metrizable.
Suppose for contradiction that the conclusion of Theorem~B fails, i.e.
that any open neighborhood $V$ of $1_H$ contains a pre-loxodromic element.

This implies the following claim, which will be used several times.

\medskip

{\bf Claim 1.} For any open neighborhood $\mathcal{O}$ of $1_H$ and any natural $k$, the subset $\mathcal{O}^{1/k}$ contains a pre-loxodromic element.

\medskip

\begin{proof} There is an open neighborhood $V$ of  $1_H$ such that
$V^k\subseteq \mathcal{O}$. By assumption, $V$ contains a pre-loxodromic element.\end{proof}

\medskip

Let $d$ be a complete metric which gives the topology on $H$.
We will construct a sequence $(u_i)_{i\in \mathbb{N}}$ of pre-loxodromic elements of $H$ and a sequence
$(\mathcal{O}_i)_{i\in\mathbb{N}}$ of open neighborhoods of $1_H$ with special properties, which will finally lead to a contradiction.

Let $\mathcal{O}_1$ be an arbitrary open neighborhood of  $1_H$, say $\mathcal{O}_1=H$.

\medskip
\noindent
{\bf Step 1.}
Select a pre-loxodromic element $u_1\in \mathcal{O}_1^{1/n}$.
We set $U_1=\varphi(u_1)$ and define the numbers $K_1,M_1$ and $m_1$ as in (4.1) and (4.2).
Select an open neighborhood $\mathcal{O}_2$ of $1_H$ such that
$$
d(u_1^n,\, u_1^n\,\mathcal{O}_2^{m_1})\leqslant \frac{1}{2}.
$$

\medskip

\noindent
{\bf Step 2.} Select a pre-loxodromic element $u_2\in \mathcal{O}_2^{1/n}$.
We set $U_2=\varphi(u_2)$ and define the numbers $K_2,M_2$  and $m_2$ as in (4.1) and (4.2).
Select an open neighborhood $\mathcal{O}_3$ of $1_H$ such that
$$
\begin{array}{rl}
d(u_1^n(u_2^n)^{m_1},\, u_1^n(u_2^n\,\mathcal{O}_3^{m_2})^{m_1})\leqslant & \frac{1}{2^2},\vspace*{2mm}\\
d(u_2^n,\, u_2^n\,\mathcal{O}_3^{m_2})\leqslant & \frac{1}{2^2}.
\end{array}
$$

\medskip

\noindent
{\bf Step\, $\bold{i.}$} Suppose that in previous steps we have already selected pre-loxodromic elements $u_1, \ldots, u_{i-1}$, numbers $K_1,\ldots, K_{i-1}$ and $M_1, \ldots, M_{i-1}$
and $m_1, \ldots, m_{i-1}$,
and neighborhoods $\mathcal{O}_1, \ldots, \mathcal{O}_i$.
Select a pre-loxodromic element $u_i\in \mathcal{O}_i^{1/n}$.
We set $U_i=\varphi(u_i)$ and define the numbers $K_i,M_i$  and $m_i$ as in (4.1) and (4.2).
Select an open neighborhood $\mathcal{O}_{i+1}$ of $1_H$ such that

$$
\begin{array}{rl}
d\bigl(u_1^n(u_2^n(\cdots  (u_i^n)^{m_{i-1}}\cdots)^{m_2})^{m_1},\,\,  u_1^n(u_2^n(\cdots (u_i^n\,\mathcal{O}_{i+1}^{m_i})^{m_{i-1}} \cdots )^{m_2})^{m_1}\bigr)\leqslant & \frac{1}{2^i},\vspace*{2mm}\\

d\bigl(u_2^n(\cdots  (u_i^n)^{m_{i-1}}\cdots)^{m_2},\,\, u_2^n(\cdots (u_i^{n}\,\mathcal{O}_{i+1}^{m_i})^{m_{i-1}} \cdots )^{m_2}\bigr)\leqslant & \frac{1}{2^i},\vspace*{0mm}\\

\vdots\vspace*{0mm}\\

d\bigl(u_i^n,\,\, u_i^n\,\mathcal{O}_{i+1}^{m_i}\bigr)\leqslant & \frac{1}{2^i}.
\end{array}
$$

For any $r,i\in \mathbb{N}$, where $i\geqslant r$, we define the element
$$
w_{r,i}=u_r^n(u_{r+1}^n(\cdots u_{i-1}^n(u_i^n)^{m_{i-1}} \cdots)^{m_{r+1}})^{m_r}.
$$

\medskip

{\bf Claim 2.}
For any $r\in \mathbb{N}$, the sequence $(w_{r,i})_{i\geqslant r}$ is Cauchy
and hence converges.

\medskip

\begin{proof} Using $u_{i+1}^n\in \mathcal{O}_{i+1}$, we deduce from the last group of inequations that
$$
\begin{array}{rl}
d(w_{1,i},w_{1,i+1})\leqslant & \frac{1}{2^i},\vspace*{2mm}\\

d(w_{2,i},w_{2,i+1})\leqslant & \frac{1}{2^i},\vspace*{2mm}\\

\vdots\vspace*{0mm}\\

d(w_{i,i},w_{i,i+1})\leqslant & \frac{1}{2^i}.
\end{array}
$$
Then, for any $j> i\geqslant r$, we have
$$
d(w_{r,i},w_{r,j})\leqslant d(w_{r,i},w_{r,i+1})+\dots +d(w_{r,j-1},w_{r,j})\leqslant
\frac{1}{2^{i}}+\dots +\frac{1}{2^{j-1}}<\frac{1}{2^{i-1}}.
$$
\end{proof}

Let $w_r=\underset{i\rightarrow \infty}{\lim} w_{r,i}$. Passing to limit in the equation $w_{r,i}=u_r^nw_{r+1,i}^{m_r}$,
which holds for any $i\geqslant r+1$, we obtain
$$
w_r = u_r^nw_{r+1}^{m_r}.
$$
Now perform the same argument as in the proof of Theorem~\ref{thebigone} on the images under $\varphi$ of the $u_r$ and $w_r$ for a contradiction.

\medskip

{\bf (b)} Now we assume that $H$ is locally countably compact Hausdorff.
Suppose for contradiction that the conclusion of Theorem~B fails, i.e.
that any open neighborhood $V$ of $1_H$ contains a pre-loxodromic element.

We will construct a sequence $(u_i)_{i\in \mathbb{N}}$ of pre-loxodromic elements of $H$ and a sequence
$(\mathcal{O}_i)_{i\in\mathbb{N}}$ of open neighborhoods of $1_H$ with special properties, which will finally lead to a contradiction.

Let $\mathcal{O}_1$ be an open neighborhood of $1_H$ with $J \supseteq \mathcal{O}_1$ a countably compact subspace of $X$.

\medskip

\noindent
{\bf Step 1.} Select a pre-loxodromic element $u_1\in \mathcal{O}_1^{1/n}$.
We set $U_1=\varphi(u_1)$ and define the numbers $K_1,M_1$ and $m_1$ as in (4.1) and (4.2).
Select an open neighborhood $\mathcal{O}_2$ of $1_H$ for which
$$
\mathcal{O}_1 \supseteq u_1^n\mathcal{O}_2^{m_1}.
$$

\medskip

\noindent
{\bf Step\, $\bold{i.}$} Suppose that in previous steps we have already selected pre-loxodromic elements $u_1, \ldots, u_{i-1}$, numbers $K_1,\ldots, K_{i-1}$ and $M_1, \ldots, M_{i-1}$
and $m_1, \ldots, m_{i-1}$,
and neighborhoods $\mathcal{O}_1, \ldots, \mathcal{O}_i$.
Select a pre-loxodromic element $u_i\in \mathcal{O}_i^{1/n}$.
We set $U_i=\varphi(u_i)$ and define the numbers $K_i,M_i$  and $m_i$ as in (4.1) and (4.2). Select an open neighborhood $\mathcal{O}_{i+1}$ of $1_H$ for which
$$
\mathcal{O}_i\supseteq u_i^n\mathcal{O}_{i+1}^{m_i}.
$$

Let $J_1 = \overline{\mathcal{O}_1}\cap J$, $J_2= u_1^n\overline{\mathcal{O}_2}^{m_1} \cap J$ and for $i>2$ letting
$$
J_i = u_1^n(u_2^n(\cdots (u_{i-1}^n\overline{\mathcal{O}_{i}}^{m_{i-1}})^{m_{i-2}} \cdots )^{m_2})^{m_1} \cap J,$$
it is clear that $J \supseteq J_1 \supseteq J_2 \supseteq \cdots$ and so the sequence $\{J_i\}_{i\in \mathbb{N}}$ is a nesting sequence of nonempty closed sets in the countably compact Hausdorff space $J$. Let $w_1\in \bigcap_{i\in \mathbb{N}}J_i$.  Select $i\in \mathbb{N}$ large enough that $i>|\varphi(w_1)|$.  Select $w_{i+2} \in \overline{\mathcal{O}_{i+2}}$ for which
$$
w_1 = u_1^n(u_2^n(\cdots (u_{i+1}^n w_{i+2}^{m_{i+1}})^{m_i} \cdots )^{m_2})^{m_1}.
$$

Let $w_{i+1} = u_{i+1}^nw_{i+2}^{m_{i+1}}, w_i = u_i^nw_{i+1}^{m_i}, \ldots, w_2 = u_2^nw_3^{m_2}$.  Clearly
$w_1=u_1w_2^{m_1}$.
As in Claim~1 of the proof of Theorem~A we see that $\varphi(w_1), \ldots, \varphi(w_{i+1})$ are loxodromic and arguing as in the conclusion of the proof of Theorem~A we get $|\varphi(w_1)|\geqslant i$, a contradiction.\hfill $\Box$

\end{section}

\begin{section}{Universal acylindrical actions and the elliptic radical}\label{Universal}

%%%%In this section we define the elliptic radical $\Rad(\Ell(G))$ of a group $G$.
%%%%We show that $\Rad(\Ell(G))=\CRad(\Ell(G))$
%%%%%Preparing for this, we need to introduce new and recall some known notions.
%%%%%We also raise two problems, which might be of general interest.
%%%%As a byproduct, we show that the group $\HEG$ is acylindrically hyperbolic,
%%%%but does not admit any universal acylindrical action (see Proposition~\ref{HEG_No_Univer}).
%%%%This seems to be the first torsion-free example of such kind.

%%%%\subsection{Universal acylindrical actions}
%%%%In this paper we are basically interested in the interactions of two categories: the category $\frak{A}$ of
%%%%topological groups (including some wild groups) and the category $\frak{B}$ of groups equipped with acylindrical actions on hyperbolic spaces. Of course any group can be considered as a topological group
%%%%and any group admits an acylindrical action on a hyperbolic space (for example on a point).

%%%%The following definition seems to be necessary for adequate understanding of this interaction.

In this section we introduce the notion of the elliptic radical $\Rad_{\Ell}(G)$ of a group $G$.
This notion is important for two reasons. The first one is that it naturally appears in Theorem~\ref{metricandlcH}.
The second one is that an important class of actions - universal acylindrical actions (introduced by Osin in~\cite{Osin_1}) - can be alternatively defined via the elliptic radical.
Knowing that $G$ admits such an action can help in computing the elliptic radical of $G$.
As it was mentioned in the introduction, many interesting groups admit such an action.

In Proposition~\ref{HEG_No_Univer} we show that the the Hawaiian earring group $\HEG$ is acylindrically hyperbolic, the elliptic radical of $\HEG$ is trivial, and that $\HEG$ does not admit any universal acylindrical action. Note that the latter claim is deduced from Theorem~\ref{thebigone}. We also prove Lemma~\ref{TwoRadicals} which is used in the proof of Theorem~\ref{metricandlcH}.

\begin{definition}\label{Definition_Of_Radical} Let $G$ be a group. An element $g\in G$ is called {\it universally elliptic} if
for any hyperbolic space $S$ and any acylindrical action of $G$ on $S$ the element $g$ acts elliptically on $S$.
The set
$$
\Rad_{\Ell}(G)=\{g\in G\,|\, g\hspace*{2mm}{\text{\rm  is universally elliptic}}\}
$$
is called the {\it elliptic radical} of $G$.
\end{definition}

\noindent
\begin{example}\label{Rad_Rel_Hyp} We consider the case where $G$ is relatively hyperbolic.
\begin{enumerate}
\item[(a)] If $G$ is a relatively hyperbolic group with respect to a collection of subgroups $\{H_{\lambda}\}_{\lambda\in \Lambda}$, then
    $$
\mathcal{U}_1\subseteq \Rad_{\Ell}(G)\subseteq \mathcal{U}_1\cup \mathcal{U}_2,
$$
where $\mathcal{U}_1$ is the set of elements of finite order of $G$ and
$\mathcal{U}_2$ is the union of conjugates of all $H_{\lambda}$'s.

We prove the second inclusion; the first one is obvious.
Let $X$ be a finite relative generating set of $G$ with respect to
$\{H_{\lambda}\}_{\lambda\in \Lambda}$.
We set $\mathcal{H}=\underset{\lambda\in\Lambda}{\sqcup} H_{\lambda}$.
Then the Cayley graph $\Gamma(G,X\cup \mathcal{H})$
is hyperbolic (see~\cite[Corollary~2.54]{Osin_0}) and $G$ acts acylindrically on $\Gamma(G,X\cup \mathcal{H})$
(see~\cite[Proposition 5.2]{Osin_1}). By~\cite[Theorem 4.25]{Osin_0}, the set of elliptic elements with respect to this action lies in $\mathcal{U}_1\cup \mathcal{U}_2$. This proves the second inclusion.

\item[(b)] If $G$ is a hyperbolic group, then $$\Rad_{\Ell}(G)=\mathcal{U}_1.$$
This follows directly from (a).

%the union of conjugates of finite subgroups of $G$ together with the union of conjugates of all $H_{\lambda}$'s.
\item[(c)] If $G$ is a relatively hyperbolic group with respect to a collection of subgroups $\{H_{\lambda}\}_{\lambda\in \Lambda}$, where none of the $H_{\lambda}$'s is virtually cyclic or acylindrically hyperbolic, then
$$
\Rad_{\Ell}(G)=\mathcal{U}_1\cup \mathcal{U}_2.
$$
Indeed, if $G$ acts acylindrically on a hyperbolic space $S$, then the restriction of this action of each $H_{\lambda}$ is elliptic by \cite[Theorem 1.1]{Osin_1}. Together with (a) this completes the proof.

%It would be interesting to compute the elliptic radical for different classes of groups and
%describe it for acylindrically hyperbolic groups.
\end{enumerate}
\end{example}

\begin{remark}\label{EllipticRadicalSmall}
The following shows that, in a certain sense, the elliptic radical of an
acylindrically hyperbolic group is relatively small.
%There are at least two theoretical reasons to claim that the elliptic radical of an
%acylindrically hyperbolic group is relatively small.

(a) In~\cite{Sisto} Sisto proved that for any acylindrically hyperbolic group $G$ generated by a finite symmetrized set $X$,
the probability that the simple random walk on the Cayley graph $\Gamma(G,X)$ arrives at an element of the elliptic radical of $G$ in $n$ steps is $O(\varepsilon^n)$ for some $\varepsilon\in (0,1)$. We don't know how to generalize this statement for the case where $G$ is not finitely generated.

(b) Suppose that $G$ acts acylindrically on a $\delta$-hyperbolic space $S$.
Then any subgroup $H$ which is contained in the elliptic radical of $G$ has bounded orbits (see~\cite[Theorem 1.1]{Osin_1})
and therefore there exists a point $s\in S$ such that the diameter of the orbit $Hx$ is at most $4\delta+1$
(see~\cite[Corollary 6.7]{Osin_1}).
\end{remark}

\medskip

To go further, we recall some relevant definitions from~\cite{Osin_1}.

\begin{definition}
In~\cite[Definition 6.4]{Osin_1} Osin gives the following definition:
Let $G$ be a group. An element $g\in G$ is called {\it generalized loxodromic} if
it satisfies either of the four equivalent conditions in~\cite[Theorem 1.4]{Osin_1}.
Two of these conditions are the following:
\begin{enumerate}
\item[$(L_1)$] There exists a generating set $X$ of $G$ such that the corresponding Cayley graph $\Gamma(G,X)$
is hyperbolic, the natural action of $G$ on $\Gamma(G,X)$ is acylindrical, and $g$ is loxodromic.

\medskip

\item[$(L_2)$] There exists an acylindrical action of $G$ on a hyperbolic space such that $g$ is loxodromic.
\end{enumerate}

Because of quantifiers, we prefer to call these elements {\it existentially loxodromic}.
Comparing with our Definition~\ref{Definition_Of_Radical}, we conclude that any group $G$ is the disjoint union of the set of universally
elliptic and the set of existentially loxodromic elements.

%At the end of~\cite[Section 6]{Osin_1},
%Osin raises two fundamental questions about generalized loxodromic elements.
\end{definition}

Osin notes that, in general, generalized loxodromic elements of a group $G$ may be elliptic for some acylindrical actions of $G$
on unbounded hyperbolic spaces.
%So, any nontrivial element of $G=F(a,b)$ is generalized loxodromic and
%for any nontrivial element $g\in G$, there exists an action of $G$
%on a hyperbolic space with respect to which $g$ is Moreover, there exists a single
This happens, for instance, for $G=F_2$. However for $F_2$ there exists a {\it single} acylindrical action on a hyperbolic space with respect to which {\it all} generalized loxodromic elements of $F_2$ are loxodromic.
This justifies the following definition.

\begin{definition}\label{Definition_Univ_Acyl_Action} (see~\cite[after Definition 6.4]{Osin_1})
An action of a group $G$ on a hyperbolic space is called {\it universal acylindrical} if it is acylindrical
and all generalized loxodromic elements of $G$ are loxodromic with respect to this action.

\medskip

In other words, an action of a group $G$ on a hyperbolic space is called universal acylindrical
if it is acylindrical and the set of elliptic elements of this action coincides with $\Rad_{\Ell}(G)$.
%\marginpar{\tiny not demanded that the action is acylindrical}
\end{definition}

Many interesting groups admit such an action, see groups in
(a)-(f) from in the introduction; these groups were considered in~\cite{ABO,ABD,Osin_1}.
%In~\cite[Example 6.5]{Osin_1}, Osin notes that the mapping class group $\MCG(\Sigma_g)$
%of a closed surface of genus $g\geqslant 2$ acts universal acylindrically
%on the curve complex of $\Sigma_g$. Also, if $G$ is relatively hyperbolic with respect
%to a collection of subgroups $\{H_{\lambda}\}_{\lambda\in \Lambda}$ and none of the $H_{\lambda}$'s
%is virtually cyclic or acylindrically hyperbolic, then $G$ acts universal acylindrically
%on the relative Cayley graph $\Gamma(G,X\cup \mathcal{H})$.
In~\cite[Example 6.10]{Osin_1}, Osin gives a countable (but not finitely) generated group which does not admit
any universal acylindrical action. In~\cite{Abbott}, Abbott shows that Dunwoody's finitely generated inaccessible group also does not admit such an action.
Both proofs ultimately use the presence of elements of increasing finite order and
Lemma~3.6.
The following seems to be the first torsion-free example of this kind.

\begin{proposition}\label{HEG_No_Univer}
The group $\HEG$ is acylindrically hyperbolic, the elliptic radical
$\Rad_{\Ell}(\HEG)$ is trivial, and $\HEG$ does not admit any universal acylindrical action.
\end{proposition}

\begin{proof}
The group $\HEG$ is acylindrically hyperbolic. To see
this, we use the decomposition $\HEG=\HEG_n\ast \HEG^n$, which holds
for any $n\in \mathbb{N}$ as mentioned in Section~\ref{HE}.
Let $\Gamma_n$ be the Cayley graph of $\HEG$ with respect to the generating set $Y_n=\{a_1,a_2,\dots ,a_n\}\,\cup\, \HEG^n$. This Cayley graph is tree-like, and hence hyperbolic, and it is easy to check that $\HEG$ acts acylindrically and nonelementary on~$\Gamma_n$.  

More particularly the group $\HEG$ has relative presentation $\HEG = \langle Y_n \cup \HEG^n \mid \emptyset \cup \Rel(\HEG^n)\rangle$, where $\Rel(\HEG^n)$ is the set of all relations
of $\HEG^n$ with respect to the generating set consisting of all elements of $\HEG^n$.
This relative presentation has trivial relative Dehn function. Therefore $\HEG$ is relatively hyperbolic with respect to the subgroup $\HEG^n$ (see~\cite[Definition 2.35]{Osin_0}).
Hence, $\HEG$ is acylindrically hyperbolic with respect to $Y_n\cup \HEG^n$ (see Definition (AH$_1$)) by~\cite[Proposition~5.2]{Osin_1}.
Moreover,
$$
\Ell\,(\HEG\curvearrowright \Gamma_n)=\underset{g\in \HEG}{\bigcup}\widehat{g}\,(\HEG^n).
$$

%In Remark~6.2 we proved that $\HEG$ is acylindrically hyperbolic and, for each natural $n$,
%the group $\HEG$ acts acylindrically on the Cayley graph $\Gamma_n$, and
%$$
%\Ell\,(\HEG\curvearrowright \Gamma_n)=\underset{g\in \HEG}{\bigcup}\widehat{g}\,(\HEG^n).
%$$
To prove that $\Rad_{\Ell}(\HEG)=1$, it suffices to prove that
$$
\underset{n\in \mathbb{N}}{\cap}\, \Ell\,(\HEG\curvearrowright \Gamma_n)=1.
$$

To the contrary, let $x$ be a nontrivial element from this intersection. Then, for every $n\in \mathbb{N}$, there exists $g_n\in \HEG$ such that $g_n^{-1}xg_n\in \HEG^n$. We denote $h_n=g_n^{-1}xg_n$ and $z_n=g_n^{-1}g_{n+1}$. Then $z_n^{-1}h_nz_n=h_{n+1}$.
By malnormality of $\HEG^n$ in $\HEG=\HEG_n\ast \HEG^n$, we have $z_n\in \HEG^n$. Now we consider the element $w=\overset{\infty}{\underset{i=1}{\prod}} z_i$. For any $n$
we can write $w=u_nv_n$, where $u_n=\overset{n}{\underset{i=1}{\prod}}z_i$, $v_n=\overset{\infty}{\underset{i=n+1}{\prod}}z_i$.
Then
$$
w^{-1}h_1w=v_n^{-1}u_n^{-1}h_1u_nv_n=v_n^{-1}h_{n+1}v_n
$$ lies in $\HEG^{n+1}$ for any $n$, hence it lies in the intersection of all $\HEG^n$, i.e. $h_1=1$.
This contradicts the assumption that $x\neq 1$.

Thus, $\Rad_{\Ell}(\HEG)=1$.
%Therefore each nontrivial element of $\HEG$ is generalized loxodromic.
Suppose that $\HEG$ admits a universal acylindrical  action $\HEG\curvearrowright S$ on a hyperbolic space $S$. Then $\Ell\,(\HEG\curvearrowright S)=1$.
This contradicts Theorem~\ref{thebigone} applied to this action and the identity homomorphism $\operatorname{id}:\HEG\rightarrow \HEG$.
\end{proof}

\begin{remark}
Since free groups admit universal acylindrical actions, this gives still another proof that $\HEG$ is not a free group.
\end{remark}

\medskip

Sometimes it is more convenient to work with cobounded actions.
So, we give a variation of Definition~\ref{Definition_Of_Radical} for cobounded actions.

\begin{definition} Let $G$ be a group. An element $g\in G$ is called {\it  c-universally elliptic} if
for any hyperbolic space $S$ and any cobounded acylindrical action of $G$ on $S$ the element $g$ acts elliptically on $S$.
We set
$$
\CRad_{\Ell}(G)=\{g\in G\,|\, g\hspace*{2mm}{\text{\rm  is c-universally elliptic}}\}.
$$
%is called the {\it cobounded elliptic $radical} of $G$.
\end{definition}

%In our old Theorems A and B, we proved that
%there exists a neighborhood $V$ of $1_H$ such that $\varphi(V)\subseteq \CRad_{\Ell}(G)$.

The following lemma is used in the proof of Theorem~\ref{metricandlcH}.

\begin{lemma}\label{TwoRadicals}
For any group $G$ holds $\Rad_{\Ell}(G)=\CRad_{\Ell}(G)$.
\end{lemma}

\begin{proof} We show that $\CRad_{\Ell}(G) \subseteq \Rad_{\Ell}(G)$; the inverse inclusion is obvious.
Suppose that $g\in G\setminus \Rad_{\Ell}(G)$.
Since $g$ is not universally elliptic, $g$ is existentially loxodromic and by condition $(L_1)$ there exists
a generating set $X$ of $G$ such that the corresponding Cayley graph $\Gamma(G,X)$
is hyperbolic, the natural action of $G$ on $\Gamma(G,X)$ is acylindrical, and $g$ is loxodromic.
Since $G$ acts coboundedly on $\Gamma(G,X)$, we have $g\in G\setminus \CRad_{\Ell}(G)$.
\end{proof}

\end{section}

%\medskip

%The following problem aims to get a more precise conclusion in these theorems in special cases.
%As is mentioned above, the elliptic radical has clear description for hyperbolic and most
%relatively hyperbolic groups.
%
%\medskip
%
%\noindent
%{\bf Problem 3.} Compute (describe) explicitly the elliptic radical $\Rad_{\Ell}(G)$ for
%commonly interesting classes of groups, in particular, for mapping class groups, 3-manifold groups,
%right-angled Artin groups, and right-angled Coxeter groups.

%\begin{remark} (for us)
%Theorem 1.1 in~\cite{Osin_1} says that if a group $G$ acts acylindrically on a hyperbolic space,
%then either $G$ has bounded orbits, or is virtually cyclic or acylindrically hyperbolic.
%
%Looking at Theorem B, we conclude that either $\varphi(H)\subseteq \Ell\,(G\curvearrowright S)$,
%or $\varphi(H)$ is virtually cyclic or acylindrically hyperbolic.
%
%Thus, the questions:
%
%1) Which completely metrizable or locally compact Hausdorff topological groups admit a virtually cyclic or
%an acylindrically hyperbolic quotient?  Do we have at least one example?
%
%2) Which completely metrizable or locally compact Hausdorff topological groups are SQ-universal?
%\end{remark}

\begin{section}{Proof of Theorem~\ref{metricandlcH}}\label{ProofofTheoremB}

To complete the proof of Theorem~\ref{metricandlcH}, we need some additional statements about groups which act acylindrically on hyperbolic spaces.

\subsection{Normalizers of elliptic subgroups of acylindrically hyperbolic groups.}
We use the following two results of Osin.
%We will use this corollary in the proof of the last statement of Theorem ....

\begin{theorem}{\rm (see~\cite[Theorem 1.1]{Osin_1})}\label{Osin_trichotomy}
Let $G$ be a group acting acylindrically on a hyperbolic space. Then $G$ satisfies exactly one of the following three
conditions.

\begin{enumerate}
\item[(a)] $G$ has bounded orbits.

\item[(b)] $G$ is virtually cyclic and contains a loxodromic element.

\item[(c)] $G$ contains infinitely many loxodromic elements whose limit sets are pairwise disjoint. In this case the action of $G$ is non-elementary and $G$ is acylindrically hyperbolic.
\end{enumerate}
 \end{theorem}

Recall~\cite{Osin_1} that a subgroup $E$ of $G$ is called {\it $s$-normal} in $G$ if $|E^g\cap E|=\infty$ for
every $g\in G$. In particular, if $E$ is infinite, then $E$ is $s$-normal in $N_G(E)$.

\begin{lemma}{\rm (see~\cite[Lemma 7.1]{Osin_1})}\label{s-normality}
Let $G$ be a group acting acylindrically and non-elementary on a hyperbolic space $S$.
Then every $s$-normal subgroup of $G$ acts non-elementary on $S$.
\end{lemma}

%For a subgroup $E$ of a group $G$ we denote by $N_G(E)$ the normalizer of $E$ in $G$.

\begin{lemma}\label{Normalizer}
Let $G$ be a group acting acylindrically on a hyperbolic space $S$.
Suppose that $E$ is an infinite subgroup of $G$ which is contained in the set $\Ell(G\curvearrowright S)$.
Then the normalizer $N_G(E)$ is also contained in $\Ell(G\curvearrowright S)$.
\end{lemma}

\medskip

{\it Proof.} Suppose that $N_G(E)$ is not contained in $\Ell(G\curvearrowright S)$. Applying Theorem~\ref{Osin_trichotomy} to $N_G(E)$, we conclude that either

\begin{enumerate}
\item[1)] $N_G(E)$ contains a finite index cyclic subgroup $Z$ generated by a loxodromic element, or

\item[2)] $N_G(E)$ acts acylindrically and non-elementary on $S$.
\end{enumerate}

Consider the first case. Since $E$ is elliptic, we have $E\cap Z=1$. But two infinite subgroups of a virtually cyclic group cannot intersect trivially. A contradiction.

Consider the second case. Since $E$ is infinite, it is $s$-normal in $N_G(E)$.
Then, by~ Lemma~\ref{s-normality}, $E$ acts non-elementary on $S$.
A contradiction.\hfill $\Box$

\medskip

Recall that a subset $X$ of a group $G$ is called {\it symmetrized} if $X=X^{-1}$.
Koubi~\cite{Koubi} proved the following proposition.

\medskip

\begin{proposition} {\rm (see~\cite[Proposition 3.2]{Koubi})}\label{Koubi_Full}
Let $G$ be a group which acts by isometries on a $\delta$-hyperbolic metric space $S$.
Let $X$ be a finite symmetrized subset of $G$ such that each element of $X\cup X\cdot X$ acts elliptically on $S$.
Then there exists a point $s\in S$ such that $d(s,xs)\leqslant 100\delta$ for every $x\in X$.
\end{proposition}

Now we strengthen the proposition of Koubi in the case of acylindrical actions.

\begin{proposition}\label{GenerKoubi}
Let $G$ be a group which acts acylindricaly by isometries on a $\delta$-hyperbolic metric space $S$.
Then there exists $\mu>0$ such that the following holds.

Let $X$ be a (possibly infinite) symmetrized subset of $G$ such that each element of $X\cup X\cdot X$ acts elliptically on $S$.
Then there exists a point $s\in S$ such that $d(s,xs)\leqslant \mu$ for every $x\in X$.
\end{proposition}

\begin{proof} We assume that $X$ is infinite.
Let $R,N>0$ be the numbers for $\varepsilon=100\delta+1$ from Definition~\ref{def Bowditch}
of an acylindrical action. We may assume that $N$ is a natural number.
Let $Y$ be a finite subset of $X$ such that $|Y|=N+1$. By Proposition~\ref{Koubi_Full},
there exists $s\in S$ such that
$d(s,ys)\leqslant \varepsilon$ for every $y\in Y$. Now we take an arbitrary $x\in X$.
Again there exists $s_1\in S$ such that $d(s_1,ys_1)\leqslant \varepsilon$ for every $y\in Y\cup \{x\}$.
By acylindrical hyperbolicity, we have $d(s,s_1)<R$. Then
$$
d(s,x(s))\leqslant d(s,s_1)+d(s_1,xs_1)+d(xs_1, xs)\leqslant \varepsilon+2R.
$$
\end{proof}

\subsection{Proof of Theorem~\ref{metricandlcH}.}
Suppose that $\varphi(H)$ is not contained in $\Rad_{\Ell}(G)$. Then there exists a hyperbolic space $S$
and a cobounded acylindrical action $G\curvearrowright S$ such that $\varphi(H)$ is not contained in
$\Ell(G\curvearrowright S)$, see Lemma~\ref{TwoRadicals}.

It follows from  Proposition~\ref{weak_metricandlcH} that there exists an open symmetrized neighborhood $\mathcal{O}$ of $1_H$ with the property
$$
\varphi(\mathcal{O}^{(2)})\subseteq \Ell(G\curvearrowright S).\eqno{(7.1)}
$$
Let $\mu$ be the constant from Proposition~\ref{GenerKoubi}. By this proposition, there exists a point $s\in S$ such that $d(s,\varphi(x)s)\leqslant \mu$ for every $x\in \mathcal{O}$.
Let $R,N>0$ be the numbers such that for any two points $p, q\in S$ satisfying $d(p, q) \geqslant R$ the set
$$
\{g\in G \mid d(p, gp)\leqslant \mu\hspace*{2mm}\text{and}\hspace*{2mm}d(q, gq) \leqslant \mu\}\eqno{(7.2)}
$$
is of cardinality at most $N$.

Since $\varphi(H)$ is not contained in $\Ell(G\curvearrowright S)$, there exists a loxodromic element $\varphi(h)\in \varphi(H)$ (see Theorem~\ref{Osin_trichotomy}).
We take a natural $n$ such that $d(s,\varphi(h)^ns)>R$. Applying (7.2) to $p=s$ and $q=\varphi(h^n)s$, we obtain
$$
|\varphi(\mathcal{O})\cap \varphi(\mathcal{O})^{\varphi(h^n)}|\leqslant N.
$$
We set $U=\mathcal{O}\cap \mathcal{O}^{h^n}$. Then $U$ is an open neighborhood of $1_H$ satisfying $|\varphi(U)|\leqslant N$ and $\varphi(U)\subseteq \Ell(G\curvearrowright S)$.
We may assume that $|\varphi(U)|$ is minimal possible. Then

\bigskip

{\it for any open neighborhood $U'$ of $1_H$ with $U'\subseteq U$ we have $\varphi(U') = \varphi(U)$.}\hfill $(7.3)$

\bigskip

We show that the subgroup $V=\langle U\rangle\cdot \ker(\varphi)$ satisfies the conclusion of
Theorem~\ref{metricandlcH}. Clearly $V$ is open.
Condition (7.3) implies that $\varphi(U)$ is a subgroup. Then, for any $z\in H$, we have
$$
\varphi(V)=\varphi(U)=\varphi(U\cap V^z)\subseteq \varphi(V^z)=\varphi(V)^{\varphi(z)}.
$$
This implies that $\varphi(V)$ is normal in $\varphi(H)$; hence $V$ is normal in $H$. Moreover,
$\varphi(V)=\varphi(U)\subseteq \Ell(G\curvearrowright S)$.
Recall that we have assumed $\varphi(H)\nsubseteq \Ell(G\curvearrowright S)$. Then, by Lemma~\ref{Normalizer}, the group $\varphi(V)$ is finite.
\hfill $\Box$
%This completes the proof of the first statement.

%Now suppose additionally that $G$ is weakly residually finite and that condition 2) is satisfied, i.e.
%that $H$ contains a normal open subgroup $V$ with finite $\varphi(V)$.
%Then there exists a finite index subgroup $G_0\leqslant G$ with $G_0\cap \varphi(V)=1$.
%The subgroup $H_0=\varphi^{-1}(G_0)$ satisfies the last statement.

\medskip

\section{Proof of a strong version of Theorem~\ref{metricandlcH}}\label{ProofofTheoremBstrong}

\begin{notation}
For any subset $X$ of a group $G$ and any natural $n$,
we denote the product $\underbrace{X\cdot X\cdot \ldots \cdot X}_n$ by $X^{(n)}$.
A subset $X$ of $G$ is called {\it symmetrized} if $X=X^{-1}$.
\end{notation}

\begin{lemma}\label{Lemma_for_MCG}
Let $G$ be a group and $H$ a subgroup of $G$ of finite index $n$.
Then for any symmetrized subset $X$ of $G$ with $1\in X$, we have
$\langle H\cap X^{(2n)}\rangle=H\cap \langle  X\rangle$.
\end{lemma}

\begin{proof} We prove only the nontrivial inclusion $H\cap \langle  X\rangle\subseteq \langle H\cap X^{(2n)}\rangle$.
Let $Hg_1$, $Hg_2$, $\dots$, $Hg_n$ be all right cosets of $H$ in $G$, where $g_1=1$.
For any $x\in X$, we connect $Hg_i$ to $Hg_j$ by a directed edge with label $x$ if $Hg_ix=Hg_j$.
We denote the resulting graph by $\Gamma$.
Let $h\in H\cap \langle  X\rangle$. We consider $h$ as a word in the alphabet~$X$. Then there exists a path $p$ in $\Gamma$ that starts and ends at $H$
and has the label $h$. This path is homotopic in $\Gamma$ to a product of paths $p_1,\dots ,p_k$,
where each $p_i$ starts and ends at $H$ and has length at most $2n$.
Let $h_i$ be the element of $H$ corresponding to the label of $p_i$. Then $h=h_1\dots h_k$ and each $h_i$ lies in $H\cap X^{(2n)}$.
\end{proof}

\medskip

We need the following slight generalization of Theorem~\ref{metricandlcH} for
the proof of Theorem C.

\begin{theorem}\label{generalisation_of_B}
Let $H$ be a topological group which is either completely metrizable, or locally countably compact Hausdorff, and let $H_0$ be a finite index subgroup of $H$.
Let $\varphi:H_0\rightarrow G$ an abstract homomorphism, where $G$ is an arbitrary group.
%Let $G$ be a group which acts coboundedly and acylindrically on a hyperbolic space $S$.
Then either $\varphi(H_0)$ is contained in the elliptic radical $\Rad_{\Ell}(G)$,
or there exists an open subgroup $V\leqslant H$
such that $H_0\cap V$ normal in $H_0$ and $\varphi(H_0\cap V)$ is finite.

%for any abstract group homomorphism $\varphi:H_0 \rightarrow G$
%either $\varphi(H_0)$ acts elliptically on $S$, or there exists an open subgroup $V\leqslant H$
%such that $H_0\cap V$ normal in $H_0$ and $\varphi(H_0\cap V)$ is finite.

%In particular, for any abstract group homomorphism $\varphi:H_0 \rightarrow G$
%there exists an open subgroup $U\leqslant H$ such that $\varphi(H_0\cap U)\subseteq\Ell(G\curvearrowright S)$.
\end{theorem}

{\it Proof.}
Suppose that $\varphi(H_0)$ is not contained in $\Rad_{\Ell}(G)$. Then there exists a hyperbolic space $S$
and a cobounded acylindrical action $G\curvearrowright S$ such that $\varphi(H_0)$ is not contained in
$\Ell(G\curvearrowright S)$, see Lemma~\ref{TwoRadicals}.

First we prove that there exists an open neighborhood $\mathcal{V}$ of identity $1_H$ such that $\varphi(H_0\cap \mathcal{V})\subseteq \Ell(G\curvearrowright S)$.
The proof is basically the same as the proof of Proposition~\ref{weak_metricandlcH}.
We need only two following changes.
Assuming the contrary, we can choose $u_i$'s so that they additionally lie in $H_0$.
Moreover, we can choose $m_i$'s so that they additionally become multiplies of $|H:H_0|!$.
Then, because of equalities $w_i=u_i^{n}w_{i+1}^{m_{i}}$,
all constructed $w_i$'s lie in $H_0$, which enables us to apply $\varphi$ to this equality and get $W_i=U_i^{n}W_{i+1}^{m_{i}}$ in $G$. Analyzing these inequalities as in the case $H=\HEG$ of Theorem~\ref{thebigone}, we get a contradiction.

Thus, there exists such a neighborhood $\mathcal{V}$. Then there exists an open neighborhood $\mathcal{O}\subseteq \mathcal{\mathcal{V}}$ of identity $1_H$ such that
$\mathcal{O}^{(2)}\subseteq \mathcal{\mathcal{V}}$. Then
$\varphi((H_0\cap \mathcal{O})^{(2)})\subseteq \Ell(G\curvearrowright S)$.
Arguing as in the proof of Theorem~\ref{metricandlcH} starting from (7.1), we obtain that
there exists a {\it relatively} open subgroup $A\leqslant H_0$ such that
$\varphi(A)$ is finite and $A$ is normal in $H_0$. Let $\widetilde{A}$ be an open subset of $H$ such that $H_0\cap \widetilde{A}=A$.

\medskip

Let $k$ be the index of $H_0$ in $H$.
Let $A_1\subseteq H$ be a symmetrized open neighborhood of $1_H$ such that $A_1^{(2k)}\subseteq \widetilde{A}$.
Then, using Lemma~\ref{Lemma_for_MCG}, we have
$$
H_0\cap \langle  A_1\rangle = \langle H_0\cap A_1^{(2k)}\rangle\subseteq \langle H_0\cap \widetilde{A}\rangle=A.
$$
We set $V=\langle A_1\rangle$. Then $V$ is an open subgroup of $H$ such that $\varphi(H_0\cap V)$ is finite
and $H_0\cap V$ is normal in $H_0$.

\hfill $\Box$

\end{section}

\begin{section}{Homomorphisms to the Mapping Class Groups} \label{MCG}

For $(g,b)\in \mathbb{N}_0^2$, let $\Sigma=\Sigma_{g,b}$ be the compact orientable surface of genus $g$ with $b$ boundary components. Let $\MCG(\Sigma)$ be the mapping class group of $\Sigma$ (we recall
the definition and some relevant results about this group below).
%For any topological group $H$, let $H^0$
%be the connected component of $1_H$. It is well known that $H^{0}$ is a closed subgroup of $H$.
%If $n=0$, we will skip the corresponding subscript and write $\Sigma_{g,b}$ instead of $\Sigma_{g,b,0}$.
%In this section we prove the following theorem.

%\begin{theorem}\label{MCG_1}
%Let $H$ be a topological group which is either completely metrizable or locally countably compact Hausdorff.
%Then for any homomorphism $\varphi:H\rightarrow \MCG(\Sigma_{g,b})$
%%almost factors through the canonical projection $H\rightarrow H/H_0$.
%there exists an open subgroup $V$ of $H$ such that $\varphi(V)$ is a finite subgroup of $\MCG(\Sigma_{g,b})$.
%\end{theorem}

%\marginpar{\tiny I think that your note must be replaced: The proofs of Kramer's statements C,D for the case of locally compact groups implicitly contain even a stronger claim that the open subgroup $V$ can be chosen to contain $H^{\circ}$. The proof of statement $D$ does not use the compactness of $H/H^{\circ}$. Instead, it uses Dantzig's theorem.}

We will prove the following slight generalization of Theorem B.
Note that we do not know how to prove Theorem B without using this generalization.

\begin{theorem}\label{MCG_2}
Let $H$ be a topological group which is either completely metrizable or locally countably compact Hausdorff.
Let $H_0$ be a finite index subgroup of $H$.
Let $\Sigma$ be a connected compact surface (possibly non-orientable and possibly with boundary components).
Then for any homomorphism $\varphi:H_0\rightarrow \MCG(\Sigma)$
%almost factors through the canonical projection $H\rightarrow H/H_0$.
there exists an open subgroup $V$ of $H$ such that $\varphi(H_0\cap V)$ is finite
and normal in $\varphi(H_0)$.
%or $\varphi(H_0)$ acts elliptically on $S$.
%a finite subgroup of $\MCG(\Sigma)$.
\end{theorem}

We note that, in case $H$ is locally compact and one is not concerned with finite index subgroups, a slight refinement of the conclusion of Theorem \ref{MCG_1} holds.  A proof of Kramer and Varghese \cite[Proposition C, Theorem D]{KV} shows that for any homomorphism $\varphi: H \rightarrow \MCG(\Sigma_{g,b})$ one has an open subgroup $V$ of $H$, with $V$ containing the connected component of $1_H$, such that $\varphi(V)$ is finite.  Their argument uses the structure theory for locally compact groups.  This approach cannot be applied in the more general situation involving finite indices and completely metrizable and locally countably compact groups.

\subsection{The mapping class group and the curve complex of $\Sigma$.}
%We must be precise giving a definition of the mapping class group $\MCG(\Sigma)$
%(it seems that some authors use different definitions).
%Some authors use slightly different definitions of the mapping class group that cause difficulties in
%using their results. Therefore
Some authors use slightly different definitions of the mapping class group. Since we use results of
Masur, Minsky, Ivanov, and Bowditch, we also use their definition of this group. The {\it mapping class group} $\MCG(\Sigma)$ is the group of isotopy classes
of orientation preserving self-homeomorphisms of $\Sigma$, where admissible isotopies fix each component of $\partial \Sigma$ setwise. In particular, the homeomorphisms are allowed to permute the boundary components, and the classes of Dehn twists around curves parallel to boundary components are trivial. Let $\PMCG(\Sigma)$ be the subgroup of $\MCG(\Sigma)$ consisting of the classes of homeomorphisms which induce the identity permutation of the boundary components and the identity permutation of the punctures.
Note that this definition slightly differs from that given in the book of Farb and Margalit~\cite{Farb_Margalit}.

A powerful tool for studying the mapping class group is the curve complex of $\Sigma$ defined by Harvey in~\cite{Harvey}.
We need only the 1-skeleton of this complex, which is called the {\it curve graph} of $\Sigma$ and is denoted by $\C (\Sigma)$ in our paper.
Recall that the vertices of $\C (\Sigma)$ are isotopy classes of essential (i.e., neither homotopically trivial nor peripheral) simple closed curves on $\Sigma$, and two isotopy classes are joined by an edge if they contain
disjoined representatives. The mapping class group $\MCG(\Sigma)$ naturally acts on the curve graph $\C(\Sigma)$.

\begin{theorem}\label{Masur-Minsky, Bowditch} Let $\Sigma=\Sigma_{g,b}$ be an orientable surface of genus $g$ with $b\geqslant 0$ boundary components such that $3g+b-4>0$. Then the following statements hold:

\begin{enumerate}
\item[(a)] {\rm (first proved by Masur-Minsky~\cite{MM};
Bowditch~\cite{Bowditch_2} gave an alternative proof)} The curve graph $\C(\Sigma)$ is hyperbolic.

\item[(b)] {\rm (Bowditch~\cite{Bowditch})} The action $\MCG(\Sigma)\curvearrowright \C(\Sigma)$ is acylindrical.

\item[(c)] {\rm (Masur-Minsky~\cite{MM})} An element of the mapping class group $\MCG(\Sigma)$ acts on the curve graph $\C (\Sigma)$ loxodromically, i.e., with positive translation length, if and only if it is pseudo-Anosov.

\item[(d)] The action $\MCG(\Sigma)\curvearrowright \C(\Sigma)$ is cobounded.
\end{enumerate}
\end{theorem}

%{\it Proof of statement (d).}
%A separating simple closed curve $C$ in $\Sigma_{g,b}$ is defined to be
%of type $\{(g_1,b_1), (g_2,b_2)\}$ if $C$ splits $\Sigma$ into
%two surfaces homeomorphic to $\Sigma_{g_1,b_1}$ and $\Sigma_{g_2,b_2}$.
Statement (d) follows from the fact that $\MCG(\Sigma)$ acts transitively (respectively, has finitely many orbits) acting on the set
of isotopy classes of nonseparating (respectively, separated) simple closed curves of $\Sigma$.
%Indeed, these orbits are defined by the types of the connected components of $\Sigma\setminus C$.

\medskip

\begin{remark}\label{exceptional cases}
If $3g+b-4\leqslant 0$, then the groups $\PMCG(\Sigma_{g,b})$ are very simple:
$\PMCG(\Sigma_{0,b})=1$ for $b\leqslant 3$ and $\PMCG(\Sigma_{0,4})=F_2$ (see~\cite[Chapter 4]{Farb_Margalit}),
%(pages 92 and 93 of~\cite{Farb_Margalit}),
$\PMCG(\Sigma_{1,1})\cong \SL_2(\mathbb{Z})$, see~\cite[Theorem 3.6]{Farb_Margalit}.
\end{remark}

We also need the following theorem of Ivanov, which is a generalization of Thurston's theorem (see~\cite{Thurston})
on elements of the mapping class group.

\begin{theorem}\label{Ivanov} {\rm (see~\cite[Theorem 1]{Ivanov_2})} Let $\Sigma$ be a compact orientable surface.\break
For any subgroup $H$ of the mapping class group $\MCG(\Sigma)$ one of the following holds:
\begin{enumerate}
\item[(a)] $H$ is finite.

\item[(b)] $H$ is reducible (that is there is an  essential $1$-submanifold of $\Sigma$
whose isotopy class is fixed by each element of $H$).

\item[(c)] $H$ contains a pseudo-Anosov element.
\end{enumerate}
\end{theorem}

The following lemma is due to Szepeitowski (in the case $b=0$ it easily follows from~\cite{Birman_Chil}).

\begin{lemma}\label{nonorMod} {\rm (see~\cite[Lemma 3]{Szep})}
Suppose that $g+2b\geqslant 3$. Let $N_{g,b}$ be a non-orientable surface of genus $g$ with $b$ boundary components. Let $\Sigma_{g-1,2b}$ be an orientable surface which is a double cover of $N_{g,b}$.
Then there exists a natural embedding of $\MCG(N_{g,b})$ into $\MCG (\Sigma_{g-1,2b})$.
\end{lemma}

\subsection{Preparation to the proof of Theorem~\ref{MCG_2}}
\medskip

\begin{definition}
Let $H$ be a topological group which is either completely metrizable or locally countably compact Hausdorff.
A group $G$ is called {\it $H$-admissible} if for any finite index subgroup $H_0$ of $H$
and any homomorphism $\varphi: H_0\rightarrow G$, there exists an open subgroup $V$ of $H$ such that $\varphi(H_0\cap V)$ is finite.
\end{definition}

\begin{proposition}\label{extensions}
Let $H$ be a topological group which is either completely metrizable, or locally countably compact Hausdorff.
Then the following holds.

\begin{enumerate}
\item The class of $H$-admissible groups is closed under taking subgroups and under taking of overgroups
of finite index. In particular, the class of $H$-admissible groups is closed under the commensurability relation.
\item The class of $H$-admissible groups is closed under taking extensions, i.e. if $G/A=B$
and the groups $A$ and $B$ are $H$-admissible, then $G$ is $H$-admissible as well.
\item The class of $H$-admissible groups contains all finite groups and all acylindrically hyperbolic groups $G$ whose elliptic radical $\Rad_{\Ell}(G)$ contains only $H$-admissible subgroups of~$G$. In particular,
all hyperbolic groups are $H$-admissible.

\item Let $H_1$ be an open subgroup of $H$. Then $H_1$ is itself either completely metrizable or locally countably compact Hausdorff. Moreover, the class of $H_1$-admissible groups is contained in the class
    of $H$-admissible groups.
\end{enumerate}
\end{proposition}

\medskip

{\it Proof.} Statements (1) and (4) are obvious, statement (3) follows from Theorem~\ref{generalisation_of_B}.
We prove statement (2).

Let $H_0$ be a finite index subgroup of $H$ and let $\varphi:H_0\rightarrow G$ be an arbitrary
homomorphism. Let $\varphi_1:G\rightarrow B$ be the canonical homomorphism.
Consider the homomorphism $\varphi_1\circ \varphi: H_0\rightarrow B$. Since
$B$ is $H$-admissible, there exists an open subgroup $U$ of $H$ such that $\varphi_1\circ \varphi\,(H_0\cap U)$ is finite. Then there exists a finite index subgroup $H_1$ of $H_0\cap U$ such that $\varphi_1\circ \varphi\,(H_1)=1$, i.e. $\varphi(H_1)\subseteq A$. Clearly $H_1$ has finite index in $U$ and $U$ (as an open subgroup of $H$) is itself either completely metrizable or locally countably compact Hausdorff. Therefore there exists an open subgroup $V$ in $U$ such that $\varphi(H_1\cap V)$ is finite.
But
$$
|(H_0\cap V):(H_1\cap V)|=|((H_0\cap U)\cap V):(H_1\cap V)|\leqslant |(H_0\cap U):H_1|<\infty.
$$
Therefore $\varphi(H_0\cap V)$ is finite.
\hfill $\Box$

\subsection{Proof of Theorem~\ref{MCG_2}.}
Due to Lemma~\ref{nonorMod}, we may assume that $\Sigma$ is orientable, say $\Sigma=\Sigma_{g,b}$ for some $g$ and $b$. We set $G=\MCG(\Sigma_{g,b})$.
%We shall prove that $G$ is $H$-admissible and either $H_0\cap V$ is normal in $H_0$ or $\varphi(H_0)$ acts elliptically on $S$
First we prove by induction on the complexity of $\Sigma$ (defined by the parameters $(g, b)$ under the lexicographic ordering) the following statement:

{\it The group $G$ 
is $H$-admissible for each $H$ which is either completely metrizable or locally countably compact Hausdorff.}

We fix such an $H$, a finite index subgroup $H_0\leqslant H$, and a homomorphism $\varphi:H_0\rightarrow G$.
We shall show that there exists an open subgroup $V\leqslant H$ such that $\varphi(H_0\cap V)$ is finite.

If $3g+b-4\leqslant 0$, then $G$ is virtually free by Remark~\ref{exceptional cases}, and the statement follows from Proposition~\ref{extensions}.

Now we consider the case $3g+b-4>0$.
%We order the tuples $(g,b)\in \mathbb{N}_0^2$ lexicographically and induct on this order.
By Theorem~\ref{Masur-Minsky, Bowditch}, $\MCG(\Sigma)$ acts acylindrically and coboundedly on the curve graph $\C(\Sigma)$, which is hyperbolic.
Therefore, by Theorem~\ref{generalisation_of_B}, either $\varphi(H_0)$ is elliptic,
or there exists an open subgroup $U\leqslant H$
such that $H_0\cap U$ is normal in $H_0$ and $\varphi(H_0\cap U)$ is finite.
%or there exists an open subgroup $U\leqslant H$ such that $\varphi(H_0\cap U)$ is finite.

In the latter case we are done. Therefore suppose that $\varphi(H_0)$ is elliptic. If $\varphi(H_0)$
is finite, we are again done. Thus, we may assume that $\varphi(H_0)$ is infinite.
%If $\varphi(H_0\cap U)$ is finite, we are done.
%\marginpar{\tiny Finite in $G$ should be made finite in $\PMCG(\Sigma)$}
%Suppose that $\varphi(H_0\cap U)$ is infinite.
%Then $\varphi(H_0)$ is elliptic and, by Theorem~\ref{Ivanov} of Ivanov, there exists an essential
%$1$-submanifold $M$ of $\Sigma$ whose isotopy class is fixed by each element of $\varphi(H_0\cap U)$.
Then, by Theorem~\ref{Ivanov} of Ivanov, there exists an essential $1$-submanifold $M$ of $\Sigma$ whose
isotopy class is fixed by each element of $\varphi(H_0)$.
We introduce the following notations.

\begin{enumerate}
\item[(a)] Let $\Sigma_1,\dots ,\Sigma_m$ be the closures of all components of $\Sigma\setminus M$.
%Note that the elements $\varphi(H_0\cap \langle U_1\rangle)$ permute the isotopy classes of these subsurfaces.

%$\bullet$ We denote $\mathcal{Z}=\mathcal{Z}(\Sigma)$ and $\mathcal{Z}_i=\mathcal{Z}(\Sigma_i)$ for $i=1,\dots,m$.

\item[(b)] Let $\PMCG(\Sigma, M)$ denote the subgroup of $\PMCG(\Sigma)$ which fixes
the isotopy class of each component of $M$ and induces the trivial permutation of $\Sigma_i$'s.
\end{enumerate}

Then there exists a finite index subgroup $K$ of $H_0$ such that
$$
\varphi(K)\leqslant \PMCG(\Sigma, M).
$$
%Recalling that $|H:H_0|<\infty$, we deduce
%$$
%|U:K|<\infty.\eqno{(9.3)}
%$$
There is an exact sequence
$$
\{1\}\rightarrow A\rightarrow \PMCG(\Sigma, M)\rightarrow \overset{m}{\underset{i=1}{\prod}}\PMCG(\Sigma_i)\rightarrow \{1\},
$$
where $A$ is the free abelian group generated by Dehn-twists around the components of $M$.
%Recall that the subgroup $U$ of $H$ is open (and hence closed).
%Therefore $U$ is itself either completely metrizable or locally countably compact Hausdorff.
Since the complexity of each $\Sigma_i$ is smaller than the complexity of $\Sigma$,
all the groups $\MCG(\Sigma_i)$ are $H$-admissible.
By Proposition~\ref{extensions}, $\PMCG(\Sigma,M)$ is $H$-admissible.
%Using (9.2) and (9.3),
Since the index $|H:K|$ is finite and $\varphi$ maps $K$ to $\PMCG(\Sigma,M)$,
we conclude that there exists an open subgroup $V\leqslant H$ such that $\varphi(K\cap V)$ is finite.
Then $\varphi(H_0\cap V)$ is also finite, since
$
|(H_0\cap V):(K\cap V|\leqslant |H_0:K|<\infty.
$
%$$
%|(H_0\cap V):(K\cap V|=|((H_0\cap U)\cap V):(K\cap V)|\leqslant |(H_0\cap U):K|<\infty.
%$$
%Therefore $\varphi(H_0\cap V)$ is finite.

We may assume that $|\varphi(H_0\cap V)|$ is minimal possible over all open subgroups $V\leqslant H$. We prove that $\varphi(H_0\cap V)$ is normal in $\varphi(H_0)$. Let $h\in H_0$. Then $V\cap V^h$ is open in $H$ and, by minimality, we have
$$
\varphi(H_0\cap V)=\varphi(H_0\cap (V\cap V^h))\leqslant\varphi(H_0\cap V^h)=\varphi(H_0\cap V)^{\varphi(h)}.
$$
\hfill $\Box$

\medskip

\noindent
\begin{theorem}
Let $H$ be a topological group which is either completely metrizable, or locally countably compact Hausdorff.
Let $H_0$ be a finite index subgroup of $H$.
Let $G$ be a one-ended hyperbolic group.
Then for any homomorphism $\varphi:H\rightarrow \Out(G)$
there exists an open subgroup $V$ of $H$ such that $\varphi(H_0\cap V)$ is finite
and normal in $\varphi(H_0)$.
\end{theorem}

\medskip

{\it Proof.} In~\cite{Bowditch_1}, Bowditch described a canonical JSJ-decomposition
of a one-ended hyperbolic group $G$.
%Some vertex groups of this decomposition are called ``MHF subgroups''.
%To formulate a result of Levitt that we use,
This is a special finite graph of groups $\Gamma$ such that $\pi_1(\Gamma)=G$.
%The set of vertices of $\Gamma$ is the disjoint union of
We need only to know that some vertex subgroups $G_v$ of $\Gamma$ are called ``MHF subgroups''
and that each such subgroup $G_v$ is the fundamental group of a compact orbifold $\Sigma_v$.
Let $V_2$ be the set of vertices $v$ of $\Gamma$ corresponding to the MHF subgroups.

Using this decomposition, Levitt proved in~\cite[Theorem 5.3]{Levitt} that $\Out(G)$ has a finite index subgroup
of the form $\mathbb{Z}^q\times M$, where $q\in \mathbb{N}$ and $M$
fits in the exact sequence
$$
\{1\} \rightarrow A\rightarrow M\rightarrow \underset{v\in V_2}{\prod}\MCG(\Sigma_v)\rightarrow \{1\}
$$
with $A$ virtually $\mathbb{Z}^s$ for some $s\in \mathbb{N}$.

By~\cite{HarveyMaclachlan}, generalized to the non-orientable case in~\cite{Fujiwara},
each group $\MCG(\Sigma_v)$ is commensurable with the the mapping class groups $\MCG(\Sigma_v')$ of a regular compact surface $\Sigma_v'$ (see also~\cite[Section 6]{Levitt}).
By Theorem~\ref{MCG_2}, each $\MCG(\Sigma_v')$ is $H$-admissible.
Using Proposition~\ref{extensions}, we conclude that $\Out(G)$ is $H$-admissible.
Therefore there exists an open subgroup $V\leqslant H$ such that $\varphi(H_0\cap V)$ is finite.
We may assume that $|\varphi(H_0\cap V)|$ is minimal possible. Then, arguing as in the last paragraph of the proof of Theorem~\ref{MCG_2},
we conclude that $\varphi(H_0\cap V)$ is normal in $\varphi(H_0)$.
\hfill $\Box$
\end{section}

\begin{section}{General applications}\label{GeneralApplications}

Given a group $G$ and an element $g\in G$, let $\widehat{g}:G\mapsto G$ be the map given by $\widehat{g}(x)=g^{-1}xg$, $x\in G$.

\subsection{Applications to endomorphisms of the group $\HEG$}

The following lemma can be easily deduced from the Bass -- Serre theory.
It can be also proved directly by using normal forms.

\begin{lemma}\label{free_product}
Let $H$ be a subgroup of a free product $A\ast B$. If every element of $H$ is conjugate to an element of $B$,
then the whole group $H$ is conjugate to a subgroup of $B$.
\end{lemma}

Let $a_i\in \HEG$ be the element corresponding the $i$-th loop of the Hawaiian earring.
For any natural $n$, the subgroup $\HEG_n$ generated by $a_1,\dots,a_n$ is free of rank $n$.
We use the decomposition $\HEG=\HEG_n\ast \HEG^n$ from Section~\ref{HE}.

%\begin{remark}\label{remark_6.1}

\begin{corollary}\label{Eda_algebraic}
For any endomorphism $\varphi:\HEG\rightarrow \HEG$ there exists $g\in \HEG$ such that the following holds:
For any natural $n$, there exists a natural $m$ such that $$(\widehat{g}\circ\varphi)(\HEG^m)\subseteq \HEG^n.$$
In particular, for $i\geqslant m+1$, the elements $(\widehat{g}\circ\varphi)(a_i)$ considered as reduced words do not contain the letters $a_1,\dots ,a_n$.
\end{corollary}

\medskip

\begin{proof}  Let $\varphi:\HEG\rightarrow \HEG$ be an endomorphism and let $n$ be an arbitrary natural number.
In the first paragraph of the proof of Proposition~\ref{HEG_No_Univer}, we established that $\HEG$ acts acylindrically
on the hyperbolic Cayley graph $\Gamma_n$.
By Theorem A, applied to $G=\HEG$ acting on its Cayley graph~$\Gamma_n$, there exists a natural number $f(n)$ such that $$\varphi (\HEG^{f(n)})\subseteq \Ell\, (\HEG\curvearrowright \Gamma_n).\eqno{(10.1)}$$
Without loss of generality, we may assume that the function $f:\mathbb{N}\rightarrow \mathbb{N}$ is monotone increasing.
Using (10.1), we deduce
$$\varphi (\HEG^{f(n)})\subseteq \underset{g\in \HEG}{\bigcup}\widehat{g}\,(\HEG^n).$$

Thus, each element of $\varphi (\HEG^{f(n)})$ is conjugate to an element of the second factor of the free product $\HEG=\HEG_n\ast \HEG^n$.  By Lemma~\ref{free_product}, there exists $g_n\in \HEG$ such that
$$
(\widehat{g_n}\circ \varphi) (\HEG^{f(n)})\leqslant \HEG^n.\eqno{(10.2)}
$$
It remains to show that $g_n$ can be chosen independently of $n$.
Since the function $f$ is monotone increasing, we have
$$
\varphi(\HEG^{f(n+1)})\subseteq \widehat{g_n}^{-1}(\HEG^n)\, \bigcap \,\,\widehat{g_{n+1}}^{-1}(\HEG^n).
$$

{\it Case 1.} Suppose that $\varphi(\HEG^{f(n+1)})\neq 1$.
Then, since $\HEG^n$ is malnormal in the free product $\HEG=\HEG_n\ast \HEG^n$, we have
$g_{n+1}g_n^{-1}\in \HEG^n$.
Therefore there exists $z_n\in \HEG^n$ such that $g_{n+1}=z_ng_n$.

\medskip

{\it Case 2.} Suppose that $\varphi(\HEG^{f(n+1)})= 1$.
Then we redefine $g_{n+1}$ by setting $g_{n+1}=g_n$.

\medskip

Thus, we may assume that for any $n$, there exists $z_n\in \HEG^n$ such that $g_{n+1}=z_ng_n$. We set $g=(\dots z_2z_1)g_1$. Then (10.2) implies that
$$
(\widehat{g}\circ \varphi) (\HEG^{f(n)})\leqslant \HEG^n
$$
for any $n$.
\end{proof}

\medskip

The following corollary follows straightforwardly from Corollary~\ref{Eda_algebraic}.
Recall that $\HEG$ is a topological group with respect to the topology $\mathcal{T}$ described in Section~2.

\begin{corollary}\label{App1}
Any endomorphism of the abstract group $\HEG$ is continuous with respect to the topology $\mathcal{T}$
of the topological group $\HEG$.
\end{corollary}

The following corollary also follows straightforward from Corollary~\ref{Eda_algebraic}.
It was first proved by Eda in~\cite[Corollary 2.11]{Ed1}.

\begin{corollary}\label{App2}  {\rm (see~\cite{Ed1})}
For any endomorphism $\varphi:\HEG\rightarrow \HEG$ there exists $g\in \HEG$ such that
the endomorphism $\widehat{g}\circ\varphi$ is induced by a continuous map from the
Hawaiian earring $\operatorname{E}$ to itself preserving the basepoint $(0,0)$.
\end{corollary}

\subsection{Applications to homomorphisms to relatively hyperbolic groups}

We use the terminology of the monograph~\cite{Osin_0}.
Let $G$ be a relatively hyperbolic group with respect to a collection of subgroups $\{H_{\lambda}\}_{\lambda\in \Lambda}$.
Let $X$ be a finite relative generating set of $G$ with respect to $\{H_{\lambda}\}_{\lambda\in \Lambda}$
and let $\Gamma=\Gamma(G,X\cup \mathcal{H})$ be the right Cayley graph of $G$ with respect to the generating set $X\cup \mathcal{H}$, where
$$
\mathcal{H}=\underset{\lambda\in \Lambda}\cup H_{\lambda}.
$$
It is proved in~\cite{Osin_0} that the elements of $G$ are either elliptic or loxodromic with respect to
the left action of $G$ on $\Gamma(G,X\cup \mathcal{H})$. This classification of elements of $G$ does not depend
on the choice of the finite set $X$. Moreover, the elliptic elements of $G$ are precisely those which have finite order or conjugate to elements of $\mathcal{H}$.

We call a subgroup $H$ of $G$ {\it parabolic} if it is conjugate to a subgroup of $H_{\lambda}$ for some $\lambda\in \Lambda$. Otherwise we call $H$ {\it non-parabolic}.

\begin{lemma} {\rm (see~\cite[Lemma~2.9]{Bog_Bux})}\label{BB}
Let $H$ be a subgroup of a relatively hyperbolic group $G$.
Then $H$ contains a loxodromic element if and only if
$H$ is infinite and non-parabolic.
\end{lemma}

\begin{proposition}\label{RelHypAcylindrical}{\rm (see~\cite[Proposition~5.2]{Osin_1})}
Let $G$ be a relatively hyperbolic group with respect to a collection of subgroups $\{H_{\lambda}\}_{\lambda\in \Lambda}$. Let $X$ be a finite relative generating set of $G$ with respect to $\{H_{\lambda}\}_{\lambda\in \Lambda}$.
Then the canonical action of $G$ on the Cayley graph $\Gamma(G, X\cup \mathcal{H})$ is acylindrical.
\end{proposition}

\begin{corollary}\label{App3} Let $G$ be a relatively hyperbolic group with respect to a collection of subgroups $\{H_{\lambda}\}_{\lambda\in \Lambda}$.
Let $\varphi:\HEG\rightarrow G$ be a homomorphism.
Then there exists $n\in \mathbb{N}$ such that the subgroup $\varphi(\HEG^n)$ is finite or parabolic.
\end{corollary}

\begin{proof}  Let $X$ be a finite relative generating set of $G$ with respect to $\{H_{\lambda}\}_{\lambda\in \Lambda}$. By Proposition~\ref{RelHypAcylindrical}, $G$ acts acylindrically on the
Cayley graph $\Gamma=\Gamma(G, X\cup \mathcal{H})$.
By Theorem~\ref{thebigone}, there exists a natural number $m$ such that $\varphi (\HEG^m)\subseteq \Ell(G\curvearrowright \Gamma)$.
Thus, $\varphi (\HEG^m)$ does not contain loxodromic elements. By Lemma~\ref{BB}, $\varphi (\HEG^m)$ is finite or parabolic.
\end{proof}

In the following corollary we obtain a stronger conclusion for a broader class of groups $H$.
The proof directly follows from Proposition~\ref{RelHypAcylindrical}, Lemma~\ref{BB} and Theorem~\ref{metricandlcH}.

\begin{corollary}\label{App3.5} Let $G$ be a relatively hyperbolic group with respect to a collection of subgroups $\{H_{\lambda}\}_{\lambda\in \Lambda}$.
Let $\varphi:H \rightarrow G$ be a homomorphism with $H$ either completely metrizable or locally countably compact Hausdorff.
Then either $\varphi(H)$ is parabolic, or there exists a normal open subgroup $V\leqslant H$ with finite $\varphi(V)$.
\end{corollary}

\subsection{Applications to homomorphisms
%from certain topological groups
to fundamental groups of graphs of groups}

For a graph $\Gamma$ we denote by $\Gamma^{0}$ the set of its vertices and by $\Gamma^1$ the set of its edges.
For any edge $e\in \Gamma^1$, let $e_{-}$ be the initial vertex of $e$ and $e_{+}$ be the terminal vertex of $e$.
Let $e^{-1}$ be the edge inverse to $e$.

We recall some terminology from the book~\cite{Serre}.
A {\it graph of groups} is a pair $\mathcal{G}=(\Gamma, \mathbb{G})$,
where $\Gamma$ is a connected graph and $\mathbb{G}$ is a triple consisting of groups $G_v$, where $v\in \Gamma^0$, of groups $G_e$ where $e\in \Gamma^1$ (we require that $G_{e}=G_{e^{-1}}$), and of embeddings $\alpha_e:G_e\rightarrow G_{e_{-}}$ and $\omega_e:G_e\rightarrow G_{e_{+}}$ where $e\in \Gamma^1$ (we require that $\alpha_e=\omega_{e^{-1}}$). A graph of groups $(\Gamma,\mathbb{G})$  is called {\it finite} if $\Gamma$ is finite.

Let $\Delta$ be a maximal subtree of $\Gamma$. The fundamental group $\pi_1(\Gamma,\mathbb{G},\Delta)$ of the graph of groups $(\Gamma,\mathbb{G})$ is the group generated by the elements of $G_v$ $(v\in \Gamma^{1})$ and the elements $t_e$ $(e\in \Gamma^1)$,
subject to the relations
$$t_e^{-1}\alpha_e(g)t_e=\omega_e(g),\hspace*{3mm} g\in G_e,\, e\in \Gamma^0,$$
$$t_{e^{-1}}=t_e^{-1},\hspace*{3mm} e\in \Gamma^0,$$
$$t_e=1,\hspace*{3mm} e\in \Delta^1.$$
The group $\pi_1(\Gamma,\mathbb{G},\Delta)$ is independent, up to isomorphism, of the choice of the maximal tree
$\Delta$.
The subgroups $G_v$ of the group $G=\pi_1(\Gamma,\mathbb{G},\Delta)$ are called the {\it vertex subgroups} of $G$.
Particular cases of fundamental groups of graph of groups are amalgamated products and HNN extensions.

The group $G$ acts by left multiplication on the Bass-Serre tree $T$ associated with $(\Gamma, \mathbb{G})$, see~\cite{Serre}.
Recall that the vertices of $T$ are the cosets $gG_v$, where $g$ runs over $G$ and $G_v$ runs over vertex subgroups of $G$. It is well known that each element of a group acting on a tree has either a nonempty fixed
points set (which is a subtree), or an invariant axis where it acts by translation through a nonzero distance. In the first case the element is called elliptic, in the second one loxodromic.

\medskip

Our nearest aim is to give a condition which provides the acylindricity of a group action on a simplicial tree.
The following definition is very similar to the definition of Sela of $k$-acylindricity
and to the definition of Bowditch of acylindricity.

\begin{definition} Let $G$ be a group acting on a simplicial tree $T$. Let $k$ and $n$ be natural numbers.
We say that $G$ acts on $T$\, {\it $(k,n)$-acylindrically}
if the pointwise stabilizer of any segment of $T$ of length $k$ contains at most $n$ elements.

%We say that $G$ acts on $T$ {\it weakly acylindrically} if
%$G$ acts on $T$\, $(k,n)$-acylindrically for some $k$ and $n$.
\end{definition}

%Recall that this amalgamation $G=A\ast_CB$ is called {\it acylindrical}
%in the sense of Sela if there exists a natural $k$
%such that the stabilizer of any segment of length $k$ is finite.
%We also say that $G$ acts on $T$ acylindrically (or $k$-acylindrically) {\it in the sense of Sela}.

\noindent
\begin{lemma}\label{number of roots lox}  Let $G$ be a group that acts $(k,n)$-acylindrically on a simplicial
tree~$T$. Then for any loxodromic element $a\in G$ and any nonzero integer $m$,
there exist at most $n$ elements $b$ satisfying $a=b^m$.
\end{lemma}

%\medskip

\begin{proof} Let $b$ and $b_1$ be two loxodromic elements satisfying $b^m=b_1^m$. Then the axes of $b$ and $b_1$ coincide, they have the same translation length and translation direction.
Then $b_1b^{-1}$ fixes this axis pointwise, hence the number of $b_1b^{-1}$ is at most $n$.\end{proof}

\noindent
\begin{proposition}\label{equiv_acylindric}
An action of a group $G$ on a simplicial tree $T$ is $(k,n)$-acylindrical for some $k,n$ if and only if it is acylindrical in the sense of Definition~\ref{def Bowditch}.
\end{proposition}

%\medskip

\begin{proof}  Suppose that $G$ acts on $T$ $(k,n)$-acylindrically.
We prove that $G$ acts on $T$ acylindrically in the sense of Definition~\ref{def Bowditch}.
For that we check that for any integer $\sigma>1$ and any two vertices $u,v\in T^0$ with $d(u,v)\geqslant k+2\sigma+2\sigma^2$ the set
$$
\mathcal{M}=\{g\in G\,|\, d(u,gu)<\sigma,\, d(v,gv)<\sigma\}
$$
contains at most $2\sigma n^2 + n$ elements.

Let $u_1,u_2,v_1,v_2$ be vertices on $[u,v]$ such that
$d(u,u_1)=\sigma$, $d(u_1,u_2)=\sigma^2$, $d(v_1,v)=\sigma$, $d(v_2,v_1)=\sigma^2$.
Thus, $[u_1,v_1]$ is the central subsegment of $[u,v]$ of length at least $k+2\sigma^2$ and $[u_2,v_2]$ is the central subsegment of $[u_1,v_1]$ of length at least $k$.
Let $g$ be an element of the set $\mathcal{M}$.

\medskip

{\it Case 1.}
Suppose that $g$ acts elliptically on $T$. Then $g$ has fixpoints on the segments
$[u,gu]$ and $[v,gv]$. Since the fixpoint set of $g$ is connected, it contains the segment $[u_1,v_1]$.
Since the length of this segment is larger than $k$, the number of such $g$ is at most $n$.

\medskip

{\it Case 2.} Now suppose that $g$ acts loxodromically on $T$. Then the translation length $t(g)$ of $g$ is at most~$\sigma$.
Moreover, the axis of $g$ has edges in $[u,gu]$ and in $[v,gv]$. Then $[u_1,v_1]$ lies on the axis of $g$.
Let $f$ be another loxodromic element from~$\mathcal{M}$.
Then $[u_1,v_1]$ lies on the axes of $g$ and $f$ and $0<t(g)\leqslant \sigma$ and $0<t(f)\leqslant \sigma$.

Let $h\in \{f,f^{-1}\}$ be the element with the same translation direction as $g$. Then $g^{t(h)}(h^{t(g)})^{-1}$ pointwise fixes $[u_2,v_2]$. Since the length of $[u_2,v_2]$ is at least $k$,
the number of possible elements $g^{t(h)}(h^{t(g)})^{-1}$ is at most $n$.
For the selected $g$ the number of possible $h^{t(g)}$ is at most $\sigma n$.
By Lemma~\ref{number of roots lox}, the number of possible $h$ is at most $\sigma n^2$. Hence the number
of possible $f$ is at most $2\sigma n^2$ and we are done.

The other implication is trivial.
\end{proof}
\begin{remark}
The definition of the $(k,n)$-acylindricity can be written in purely algebraic terms.
For example, if $G=A\ast_CB$, then the action of $G$
on the associated Bass-Serre tree is $(2,n)$-acylindrical if and only if $|C^g\cap C|<\infty$ for every $g\in G\setminus C$.
\end{remark}

%\begin{corollary}\label{App5}
%Let $G$ be the fundamental group of a finite graph of groups $(\Gamma,\mathbb{G})$. Suppose that the action of $G$
%on the Bass-Serre tree associated with $(\Gamma,\mathbb{G})$ is $(k,n)$-acylindrical for some $k,n\in \mathbb{N}$.
%Then for any abstract group homomorphism $\varphi:\HEG\rightarrow G$, there exists a natural number $m$
%such that $\varphi (\HEG^m)$ is conjugate to a subgroup of a vertex subgroup of $G$.
%\end{corollary}

%\medskip

%\begin{proof}  Let $T$ be the Bass-Serre tree associated with the graph of groups $(\Gamma,\mathbb{G})$.
%Since $\Gamma$ is finite,  $G$ acts coboundedly on $T$. By Lemma~\ref{equiv_acylindric},
%this action is acylindrical, and the statement directly follows from Theorem~\ref{thebigone}
%and Lemma~\ref{product_elliptic}.
%\end{proof}

\medskip

\begin{lemma}\label{FixedpointLemma} Suppose that $G$ is a group acting acylindrically on a simplicial tree $T$.
If $G$ does not contain loxodromic elements, then there exists a vertex or an edge of $T$
which, up to inversion, is fixed by $G$.
\end{lemma}

\medskip

{\it Proof.}  By Theorem~\ref{Osin_trichotomy}, $G$ has an orbit of finite diameter, say $Gx$, where $x$ is a vertex of $T$. Let $\Gamma$ be the tree spanned by $Gx$. Then $\Gamma$ has finite diameter and is $G$-invariant.
A vertex of $\Gamma$ is called {\it outer} if it has degree 1.
An edge of $\Gamma$ is called {\it outer} if at least one of its endpoints is outer.
All other vertices and edges of $\Gamma$ are called {\it inner}.
Let $\Gamma_1$ be the subset of $\Gamma$ consisting of all inner vertices and edges.
If $\Gamma$ contains at least 2 edges, then $\Gamma_1$ is a nonempty subtree of $\Gamma$.
Moreover, $\Gamma_1$ has smaller diameter than $\Gamma$ and is $G$-invariant.
Inducting on the diameter of $\Gamma$, we complete the proof. \hfill $\Box$

\medskip

The following two corollaries follow directly from Theorem~\ref{metricandlcH} (respectively, from Theorem~\ref{thebigone}) with the help of
Proposition~\ref{equiv_acylindric} and Lemma~\ref{FixedpointLemma}.

\begin{corollary}\label{App6}   Let $H$ be a topological group which is either completely metrizable, or locally countably compact Hausdorff.
Let $G$ be the fundamental group of a (possibly infinite) graph of groups $(\Gamma,\mathbb{G})$. Suppose that the action of $G$
on the Bass-Serre tree associated with $(\Gamma,\mathbb{G})$ is $(k,n)$-acylindrical for some $k,n\in \mathbb{N}$.
Then for any abstract group homomorphism $\varphi:H \rightarrow G$ either $\varphi(H)$ is conjugate into
a vertex subgroup of $G$ or there exists a normal open subgroup $V\leqslant H$ such that $\varphi (V)$ is finite.
%conjugate into a vertex subgroup of $G$.\marginpar{\tiny $V$ was
%an open neighborhood of $1_H$. Now need more fine arguments, I'll write them.}
\end{corollary}

\begin{corollary}\label{App7}   Let $H$ be a topological group which is either completely metrizable, or locally countably compact Hausdorff.
Let $G$ be the fundamental group of a (possibly infinite) graph of groups $(\Gamma,\mathbb{G})$. Suppose that the action of $G$
on the Bass-Serre tree associated with $(\Gamma,\mathbb{G})$ is $(k,n)$-acylindrical for some $k,n\in \mathbb{N}$.
Then for any abstract group homomorphism $\varphi:\HEG \rightarrow G$
there exists a natural number $m$ such that $\varphi (\HEG^m)$ is a subgroup of $G$ consisting of only elliptic elements with respect to this action.
\end{corollary}

\end{section}

\begin{section}{Combination theorems for automatic continuity}\label{Automatic_Continuity}
From Corollaries \ref{App3} and \ref{App3.5} one obtains new results regarding automatic continuity.  Recall that a group $G$ is \emph{n-slender} if for every abstract homomorphism $\varphi: \HEG \rightarrow G$ there exists $n\in \mathbb{N}$ such that $\HEG^n \leqslant \ker(\varphi)$.  Similarly $G$ is \emph{cm-slender} (respectively \emph{lcH-slender}) if every abstract group homomorphism from a completely metrizable (respectively locally compact Hausdorff) topological group to $G$ has open kernel \cite{CC}.
The abbreviation \emph{lccH-slender} stands for locally countably compact slenderness.
These can be considered as a strong type of automatic continuity: one can endow a cm-slender group $G$ with the discrete topology and assert that any abstract group homomorphism from a completely metrizable topological group to $G$ is continuous.  Free groups, Thompson's group $F$, Baumslag-Solitar groups, $\mathbb{Z}[\frac{1}{m}]$, torsion-free word hyperbolic groups, graph products of such groups and numerous other groups are now known to satisfy these notions of slenderness.  A group which is either n-slender, cm-slender, or lccH-slender must be torsion-free.

%\begin{bigtheorem}\label{slendernessrelhyp}
%\begin{corollary}\label{slendernessrelhyp}
\medskip
\noindent
{\bf Theorem~\ref{slendernessrelhyp}.}
{\it If $G$ is torsion-free and relatively hyperbolic with respect to
a collection $\{H_{\lambda}\}_{\lambda \in \Lambda}$ of n-slender (respectively cm-slender, lcH-slender, lccH-slender) subgroups, then $G$ is also n-slender (respectively cm-slender, lcH-slenderm, lccH-slender).
}

\begin{proof}  Suppose $G$ is torsion-free, relatively hyperbolic to a collection of n-slender subgroups. Let $\varphi: \HEG \rightarrow G$ be a homomorphism. By Corollary \ref{App3}, there exists some $n\in \mathbb{N}$ for which $\varphi(\HEG^n)$ is parabolic.  Then the restriction $\varphi\upharpoonright \HEG^n$ is a map from an isomorphic copy of $\HEG$ to an n-slender group, so there exists an $m \geqslant n$ for which $\varphi\upharpoonright \HEG^m$ is the trivial map.  Thus $G$ is n-slender. The proof for other types of slenderness is similar.
\end{proof}

We obtain the following from Corollaries~\ref{App6} and~\ref{App7}:

%\begin{corollary}\label{slendernessgraph}
\medskip
\noindent
{\bf Theorem~\ref{slendernessgraph}.}
{\it Suppose that $G$ is the fundamental group of a (possibly infinite) graph of groups $(\Gamma,\mathbb{G})$ where each vertex group $G_v$ is n-slender (respectively cm-slender, lcH-slender, lccH-slender) and that the action of $G$ on the Bass-Serre tree associated with $(\Gamma,\mathbb{G})$ is $(k,n)$-acylindrical for some $k,n\in \mathbb{N}$.  Then $G$ is n-slender (respectively cm-slender, lcH-slender, lccH-slender).
}

\end{section}

\end{document}